\documentclass[11pt]{amsart}
\usepackage[verbose,letterpaper,tmargin=3cm,bmargin=3cm,lmargin=2.5cm,rmargin=2.5cm]{geometry}
\usepackage{amsmath,amssymb,amsthm,mathrsfs,latexsym, mathtools}
\usepackage[T1]{fontenc}
\usepackage[utf8]{inputenc}
\usepackage{enumitem}
\usepackage{bm}
\usepackage[usenames, dvipsnames]{color}
\usepackage{listings}
\usepackage{siunitx}
\usepackage[draft=false,linktocpage=true]{hyperref}
\usepackage[capitalise]{cleveref}
\usepackage{comment}
\usepackage{calligra}
\usepackage{setspace}
\usepackage[final]{microtype}
\usepackage{tikz,tikz-cd}
\usetikzlibrary{trees,calc, arrows, decorations.markings}
\usepackage{graphicx}
\usepackage{wrapfig}
\usepackage{float}
\usepackage{caption}
\usepackage{subcaption}
\numberwithin{equation}{section}
\newtheorem{theorem}{Theorem}[section]
\newtheorem{theodef}[theorem]{Theorem/Definition}
\newtheorem{proposition}[theorem]{Proposition}
\newtheorem{lemma}[theorem]{Lemma}
\newtheorem{definition}[theorem]{Definition}
\newtheorem{corollary}[theorem]{Corollary}

\newtheorem{remark}[theorem]{Remark}

\DeclareMathOperator{\Hom}{Hom}

\DeclareMathOperator{\End}{End}

\DeclareMathOperator{\iHom}{\mathscr{H}\text{\kern -3pt {\calligra\large om}}\,}
\DeclareMathOperator{\iEnd}{\mathscr{E}\text{\kern -3pt {\calligra\large nd}}\,}
\DeclareMathOperator{\iExt}{\mathscr{E}\text{\kern -3pt {\calligra\large xt}}\,}
\DeclareMathOperator{\iTor}{\mathscr{T}\text{\kern -3pt {\calligra\large or}}\,}

\DeclareMathOperator{\colim}{colim}
\DeclareMathOperator{\HCH}{HH^{\bm{\cdot}}}

\tikzset{%
    symbol/.style={%
        draw=none,
        every to/.append style={%
            edge node={node [sloped, allow upside down, auto=false]{$#1$}}}
    }
}

\newcommand{\Mod}[1]{#1\textrm{-}\textsf{Mod}}
\newcommand{\RMod}[1]{#1\textrm{-}\textsf{RMod}}

\newcommand{\Spec}[1]{\mathrm{Spec} \, #1}

\newcommand\opposite[1]{{#1}^{\mathrm{op}}}
\newcommand{\id}{\mathrm{id}}
\newcommand{\QC}[1]{\mathrm{QCoh(#1)}}

\newcommand\rightthreearrow{%
        \mathrel{\vcenter{\mathsurround0pt
                \ialign{##\crcr
                        \noalign{\nointerlineskip}$\rightarrow$\crcr
                        \noalign{\nointerlineskip}$\rightarrow$\crcr
                        \noalign{\nointerlineskip}$\rightarrow$\crcr
                }%
        }}%
    }
\newcommand\righttwoarrow{%
        \mathrel{\vcenter{\mathsurround0pt
                \ialign{##\crcr
                        \noalign{\nointerlineskip}$\rightarrow$\crcr
                        \noalign{\nointerlineskip}$\rightarrow$\crcr
                }%
        }}%
}

\title{The Derived Ring of Differential Operators}
\author{Andy Jiang}
\address{Department of Mathematics,University of Michigan, 530 Church Street,
  Ann Arbor, MI 48109}
\email{ndjiang@umich.edu}

\begin{document}
\maketitle

\begin{abstract}
  By reading a standard formula for the ring of Grothendieck differential operators in a derived way, we construct
  a derived (sheaf of) ring of Grothendieck differential operators for Noetherian schemes $X$ separated and finite-type over a base $S$,
  when the map $X \to S$ is finite tor-amplitude.
  Using this ring of differential operators, we (re-)develop the theory of $D$-modules from scratch
  and show an equivalence of categories between $D$-modules using our definition and crystals over the infinitesimal site.
\end{abstract}

\section*{Introduction}
\textbf{Motivation}:
Modules over the ring of differential operators, or $D$-modules for short, were first studied following
ideas of Mikio Sato. $D$-modules provide an algebraic framework in which one could study differential equations
and constitute a vast generalization of the theory of flat connections on vector bundles. Since then, $D$-modules
have become an invaluable tool in algebraic geometry and representation theory.

For $X=\Spec{A}$, a smooth affine variety over a field $k$, the Grothendieck ring of differential operators on $X$ relative to $k$, $D_{X/k}$,
is the increasing union
\[D_{X/k} := \bigcup_{n \ge 0} {D^{(n)}} \subseteq \Hom_k(A,A)\]
where $D^{(n)} \subseteq \Hom_k(A,A)$ is defined inductively by
\[D^{(-1)}=0\]
and
\[D^{(n)}=\{f \in \Hom_k(A,A) | \forall a \in A,  [f,a] \in D^{(n-1)}\}\]
($a \in A$ is thought of as an element $\Hom_k(A,A)$ via multiplication by $a$)

By a $D$-module on $X$, we then refer to a module over the ring $D_{X/k}$. For $X$ a general smooth variety,
we can glue this definition Zariski-locally: $D_{X/k}$ becomes a quasicoherent sheaf of algebras (though with two different
actions of the structure sheaf--on the left and right), and a $D_{X/k}$-module refers to
a quasicoherent sheaf with an action of $D_{X/k}$. If one only studies $D$-modules
on smooth varieties, such a definition will suffice. However, over singular varieties, the same definition will lead to many unpleasant properties.

For this reason, two alternative definitions were proposed for studying $D$-modules on singular varieties over a field $k$.
The first stems from Kashiwara's equivalence,
which says that if $Z$ is embedded in a smooth variety $X$ via a closed immersion,
then the category of $D$-modules on $X$ supported on $Z$ is independent
of the choice of $X$ and in the cases where $Z$ is smooth agree with the category of $D$-modules on $Z$. Therefore, when $Z$ is singular,
one can define $D$-modules on $Z$ as $D$-modules on $X$ supported on $Z$, after
a closed embedding $Z \xhookrightarrow{} X$ into a smooth ambient variety $X$ has been chosen.

A more intrinsic definition was given by Grothendieck.
Namely, for any variety $X$, smooth or singular, we can consider the (small) site of infinitesimal thickenings $U \to T$ where $U$ varies over open subsets of $X$.
A crystal (for the infinitesimal site) on $X$ is then (roughly speaking)
the data of a quasicoherent $\mathcal{O}_T$ module $\mathcal{F}_T$ for each thickening $U \to T$, such that for any morphism
of thickenings in the infinitesimal site, the natural map
\[f^*\mathcal{F}(T') \to \mathcal{F}(T)\]
is an isomorphism. It is possible to show these two definitions agree (in the sense of an equivalence of categories), giving a consistent notion of a $D$-module
on a singular variety. Nevertheless, one may ask whether there is a third approach, more similar to the definition in the smooth
setting, where we can explicit construct a quasicoherent sheaf of algebras $D_X$ on a singular variety $X$ such
that $D_X$ modules will give the same category of $D$-modules as the two approachs mentioned above.

In the present paper, we will show that this is indeed possible, and that the correct definition
of $D_X$ will simply the derived version of one of the standard definitions for $D_X$ in the smooth setting.
Let us now indicate which definition of $D_X$ we intend to derive. For simplicity, we will assume $X=\Spec{A}$
is an affine underived Noetherian scheme. In this setting, it is well known that in the case $A$ is smooth, there is an isomorphism
\[D_A \cong \colim_n(\Hom_A((A \otimes_k A)/I_{\Delta}^n,A))\]
where $I_{\Delta}$ is the kernel of the multiplication map $\mu_A : A \otimes_k A \to A$, and
the formula is the same whether we read it in a derived way or not.
We should note that it is not extremely clear what the algebra structure on $D_A$ is from this isomorphism.
In the case $A$ is singular, we will simply take the same definition, but now require that we read it in a fully derived
manner (see \textit{Conventions} below).

We note the following isomorphisms which follow simply from tensor-hom adjunction.
\begin{equation*}
  \begin{split}
    \colim_n(\Hom_A((A \otimes_k A)/I_{\Delta}^n,A)) 
    &\cong \colim_n(\Hom_A((A \otimes_k A) \otimes_{A \otimes_k A} (A \otimes_k A)/I_{\Delta}^n,A)) \\
    &\cong \colim_n(\Hom_{A \otimes_k A}((A \otimes_k A)/I_{\Delta}^n,\Hom_A(A \otimes A,A)) \\
    &\cong \colim_n(\Hom_{A \otimes_k A}((A \otimes_k A)/I_{\Delta}^n,\Hom_k(A,A))) \\
    &\cong \Gamma_{\Delta}(\Hom_k(A,A))
  \end{split}
\end{equation*}
where $\Gamma_{\Delta}$ means taking local-cohomology at the diagonal of $\Spec A$.
Note that this presentation because it makes
the algebra structure evident. We note that this formula for the ring
of differential operators can be found in Section 2.1
of \cite{SVdB}, where they also briefly study the derived ring of differential operators.

It is $\Gamma_{\Delta}(\Hom_k(A,A))$ that we will take as definition for $D_A$ (or rather
a more categorical rewriting of this). Namely, we have the following definition from \cite{previous}.
\begin{definition}[Definition 1.1 of \cite{previous}] \label{bigdef}
  For $S$ a homologically bounded spectral Noetherian base scheme. Suppose
  $p_X: X \to S$ is separated, locally almost of finite presentation and finite tor-amplitude.
  The Grothendieck sheaf of differential operator of $X$ over $S$ is defined to be
  \[D_{X/S} := \Gamma_{\Delta} \pi_1^{\times}\mathcal{O}_X \in \Gamma_{\Delta}(\QC{X \times_S X})\]
  where $\pi_1^{\times}$ is the right adjoint of the pushforward functor $\pi_{1,*}$.
  Often we will suppress $S$ from the notation and write simply $D_X$.
\end{definition}
In \cite{previous} we ignored the issue of the ring structure on $D_X$. In the current paper, we will
endow $D_X$ with a ring structure with respect to the convolution monoidal structure on $\Gamma_{\Delta}(\QC{X \times_S X})$
and study $D$-modules from this perspective.

Additional discussions on the derived ring of differential operators can be found in \cite{jack}, though the goals of these papers
are markedly different from the present paper. In \cite{Haiping}, they also discuss constructing a derived
ring of differential operators on singular varieties. However, the approach in the present paper is more intrinsic and
works in a slightly different generality. The derived ring of differential operators is also defined in \cite{CryD}. However we believe
our approach is more concrete.

\textbf{Results}:
The bulk of the text is concerned with developing a theory of $D$-modules using the definition
of the derived ring of differential operators shown above. Let us summarize the culmination of this effort.
The theorem below can be found in the text in Theorem \ref{rightdref}, Theorem \ref{bigtheorem2}, and Theorem
\ref{pftdescent}.
\begin{theorem}
  For any spectral Noetherian scheme $X$ finite tor-amplitude, locally almost of finite-presentation, and separated
  over a base spectral Noetherian scheme $S$, both homologically bounded,
  there is an ring of differential operators
  \[D_{X/S} \in \Gamma_{\Delta}(\QC{X \times_S X})\]
  such that if $X = \Spec{A}$ and $S = \Spec{k}$, then $D_{X/S} \cong \Gamma_{\Delta}\Hom_k(A,A)$.
  
  Here, \[\Gamma_{\Delta}(\QC{X \times_S X})\]
  is the subcategory of $\QC{X \times_S X}$ supported on the diagonal. It acquires a monoidal structure via the isomorphism
  \[\QC{X \times_S X} \cong \Hom_{\QC{S}}(\QC{X},\QC{X})\]
  where on the right hand side the Hom is taken in $\Mod{\QC{S}}^L$ (see Appendix A of \cite{previous})

  As $\Gamma_{\Delta}(\QC{X \times_S X})$ acts on $\QC{X}$, $D_{X/S}$ defines a monad on $\QC{X}$ and we can consider the category of $D_{X/S}$ modules
  \[\Mod{D_{X/S}}\]
  Additionally, $\Gamma_{\Delta}(\QC{X \times_S X})$ carries a natural involution via swapping the two copies of $X$. The image of $D_{X/S}$
  under this involution is called $\opposite{D_{X/S}}$. We can also consider the modules under this monad, which we call
  \[\Mod{\opposite{D_{X/S}}}\]

  Both $\Mod{D_{X/S}}$ and $\Mod{\opposite{D_{X/S}}}$ satisfies étale (in fact proper finite tor-amplitude) descent with respect to $X$ and fpqc descent with respect to $S$.
  Also, we have the following isomorphisms
  \begin{equation*}
    \begin{split}
      \Mod{\opposite{D_{X/S}}}
      &\cong \colim_{\opposite{\bm{\Delta_s}}}(\Gamma_{\Delta}(\QC{X^{n+1}}),*) \\
      &\cong \QC{X} \otimes_{\Gamma_{\Delta}(\QC{X \times X})} \QC{X} \\
    \end{split}
  \end{equation*}
  \begin{equation*}
    \begin{split}
      \Mod{D_{X/S}}
      &\cong \lim_{\bm{\Delta_s}}(\Gamma_{\Delta}(\QC{X^{n+1}}),*) \\
      &\cong \Hom_{\Gamma_{\Delta}(\QC{X \times X})}(\QC{X},\QC{X}) \\
    \end{split}
  \end{equation*}
  where the last Hom is taken in $\Mod{\QC{X \times_S X}}^L$.
\end{theorem}

Of vital importance in $D$-module theory are the pushforward and pullback functors. We define them in Section \ref{pushpullsection}. The
following is a rewriting of the beginning of Section \ref{pushpullsection}. The last claim below is clear from definitions, see Section \ref{pushpullsection}
for details.
\begin{theorem}
  Suppose $S$ is a homologically bounded spectral Noetherian scheme and $f: X \to Y$ is a map
  between schemes which are finite tor-amplitude, locally almost of finite presentation and separated over $S$. Then, there is a natural pullback functor
  \[f^+ : \Mod{D_{Y/S}} \to \Mod{D_{X/S}}\]
  that when written as a map
  \[f^+ : \lim_{\bm{\Delta_s}}(\Gamma_{\Delta}(\QC{Y^{n+1}}),*) \to \lim_{\bm{\Delta_s}}(\Gamma_{\Delta}(\QC{X^{n+1}}),*)\]
  is defined by quasicoherent pullback (upper star) termwise.

  There is dually a natural pushforward functor
  \[f_+ : \Mod{\opposite{D_{X/S}}} \to \Mod{\opposite{D_{Y/S}}}\]
  that when written as a map
  \[f_+ : \lim_{\bm{\Delta_s}}(\Gamma_{\Delta}(\QC{X^{n+1}}),*) \to \lim_{\bm{\Delta_s}}(\Gamma_{\Delta}(\QC{Y^{n+1}}),*)\]
  is defined by quasicoherent pushforward (lower-star) termwise.
  
  Both functors compose well, in the sense that if $f: X \to Y$ and $g: Y \to Z$, then
  \[f^+g^+ \cong (gf)^+\]
  and
  \[g_+f_+ \cong (gf)_+\]
\end{theorem}

In addition, we have (see Theorem/Definition \ref{transfermod} for details)
\begin{proposition}
  With the same assumptions as above, the functors $f_+$ and $f^+$ correspond to the tranfer $(D_{X/S},D_{Y/S})$-bimodule
  \[\Gamma_f(\mathcal{O}_X \boxtimes \omega_Y) \in \Gamma_f(\QC{X \times_S Y})\]
  where $\Gamma_f$ means restricting to sections supported on the graph of $X$ inside $X \times_S Y$.
\end{proposition}

We also prove a left-right switch for $D$-modules with our definitions. The following is Theorem \ref{leftrightswitch} in
the main text, combined with the discussion above that Theorem.
\begin{theorem}
  For any spectral Noetherian scheme $X$ finite tor-amplitude, locally almost of finite-presentation, and separated
  over a base spectral Noetherian scheme $S$, both homologically bounded, there is an isomorphism
  \[\Mod{D_{X/S}} \cong \Mod{\opposite{D_{X/S}}}\]
  induced by the $(\opposite{D_{X/S}},D_{X/S})$-bimodule
  \[\Gamma_{\Delta}(\omega_{X/S} \boxtimes \omega_{X/S})\]
  and the inverse is induced by the $(D_{X/S},\opposite{D_{X/S}})$-bimodule
  \[\Gamma_{\Delta}(\mathcal{O}_X \boxtimes \mathcal{O}_X)\]
\end{theorem}

Lastly in the theory of $D$-modules, we also prove a form of Kashiwara's equivalence with our definitions--this is
Theorem \ref{kashi} and Corollary \ref{kashi2} in the main text.
\begin{theorem}
  Let $X$ be a spectral Noetherian scheme which is finite tor-amplitude, locally almost of finite-presentation, and separated
  over a base spectral Noetherian scheme $S$, both homologically bounded.
  Let $z : Z \to X$ is a finite tor-amplitude closed immersion. Then, the functor
  \[z^+ : \Mod{D_{X/S}} \to \Mod{D_{Z/S}}\]
  restricts to an equivalence of categories on $\Gamma_Z(\Mod{D_{X/S}})$--the full subcategory supported on $Z$.
  Dually, the functor
  \[z_+ : \Mod{\opposite{D_{Z/S}}} \to \Mod{\opposite{D_{X/S}}}\]
  is an equivalence onto the full subcategory $\Gamma_Z(\Mod{\opposite{D_{X/S}}})$ of the codomain.
\end{theorem}

Finally, using the above we prove the following isomorphism (note that the conditions here are more restrictive than above)
stated as Theorem \ref{stackequi} in the text. For the last claim below, see Appendix \ref{cryss}.
\begin{theorem}
  Let $S$ be an underived Noetherian scheme and $X$ be an underived finite-type $S$-scheme which is finite tor-amplitude and separated over $S$,
  then there is a natural isomorphism
  \[QCoh((X/S)_{dR}) \cong \Mod{D_{X/S}}\]
  The former is also naturally isomorphic to the category of quasi-coherent crystals on the small or big infinitesimal site.
\end{theorem}

We also prove a decategorification of Proposition 4.2.5 in \cite{Ber2} in the smooth setting.
It is (one of) the main results of Section \ref{hochc}, and we leave the explanation of the notation to that section.
\begin{proposition}
  In the setting of where $X=\Spec{A}$ is affine and smooth over a base $S=\Spec{k}$
  which is discrete, we have the following isomorphism
  \begin{equation}
    D_A \cong \scalebox{1.3}{%
      $\substack{A_{180^{\circ}} \\ \otimes\\ A} \scriptstyle{\HCH(A/k)}$}
  \end{equation}
\end{proposition}

Lastly, we provide an application of the theory to recover a main result of Ben-Zvi and Nevins in \cite{Cusp}.
See Section \ref{cuspsec} or \cite{Cusp} for the relevant definitions. The following is the Theorem \ref{recover} in the text
and Theorem 1.4 of \cite{Cusp}.
\begin{theorem}
  Suppose $\tau : \tilde{X} \to X$ is a \textit{good} cuspidal quotient of \textit{good} Cohen-Macaulay varieties
  over a field $k$, then $D_{\tilde{X}}$ and $D_X$ are concentrated in degree $0$ and Morita equivalent.
\end{theorem}

\textbf{Strategy}:
Let us fix spectral Noetherian schemes $X$, $S$ homologically bounded and such that the structure map
$p_X : X \to S$ is locally almost of finite-presentation, separated, and finite tor-amplitude. At the heart of the paper is the monoidal category
\begin{equation} \label{myfav}
  \Gamma_{\Delta}(\QC{X \times_S X})
\end{equation}
of quasicoherent sheaves on $X \times X$ supported on the diagonal, with monoidal product given by convolution
(see Theorem/Definition 1.11 of \cite{previous}). We think of this category as the categorified ring of differential operators (though
it may be more familiar as quasicoherent sheaves on the infinitesimal groupoid).
This perspective is by no means new to us, and we learned of it from \cite{Ber1}. Let us nevertheless try to justify
this intuition. Consider the isomorphism
\begin{equation} \label{Ddef}
  D_{A/k} \cong \Gamma_{\Delta}(\Hom_k(A,A))
\end{equation}
for an affine scheme $\Spec{A}$ over $\Spec{k}$.
At one categorical level higher, we should expect that the categorified ring of differential operators $\mathcal{D}_{X/S}$
(for a scheme $X$ over $S$) satisfies
\begin{equation} \label{DDdef}
  \mathcal{D}_{X/S} \cong \Gamma_{\Delta}(\Hom_{\QC{S}}(\QC{X},\QC{X}))
\end{equation}
at least when $X$ is $1$-affine, in the terminology of \cite{1affine}
\footnote{
  Being $1$-affine means that categorified quasicoherent sheaves are determined by their global sections.
  We remind the reader that categorified quasicoherent sheaves are
  sheaves of categories which are quasicoherent modules over the categorified structure sheaf, $U \mapsto \QC{U}$.
  Nevertheless all schemes considered in this paper will be $1$-affine, as they are quasicompact quasiseparated spectral schemes.}.
Equation (\ref{DDdef}) tells us that
$\mathcal{D}_{X/S}$ is indeed the monoidal category (\ref{myfav}) above, justifying our intuition.
For a discussion of categorified $D$-modules,
see \cite{Ber2}. 

Let us describe our strategy for constructing the theory of $D$-modules in our framework. Roughly speaking,
we will explicitly construct a category
which we claim to be the category of $D$-modules on a scheme $X$. Afterwards, we will show that this category is monadic over $\QC{X}$
with the monad given by the quasicoherent sheaf of algebras $D_X$.
This is also the approach taken in \cite{CryD}. We will give two explicit forms of the category of $D$-modules.
One, as a colimit/limit of categories of the form
\[\Gamma_{\Delta}(\QC{X^n})\]
which are quasicoherent sheaves on the $n$-fold product of $X$ (over $S$) supported at the diagonal (this is also
the presentation in \cite{CryD}) \footnote{In the limit form, this is the
data of what is commonly called a (co)stratification}. The other form
is directly in terms of the monoidal category $\Gamma_{\Delta}(\QC{X \times_S X})$. To be precise, we have the isomorphisms
(Theorem \ref{rightdref} and Theorem \ref{bigtheorem2} respectively)
\begin{equation} \label{catisoD}
  \Mod{\opposite{D}_X} \cong \QC{X} \otimes_{\Gamma_{\Delta}(\QC{X \times X})} \QC{X} 
\end{equation}
\begin{equation} \label{catisoD2}
  \Mod{D_X} \cong \Hom_{\Gamma_{\Delta}(\QC{X \times X})}(\QC{X},\QC{X})
\end{equation}
These isomorphisms appeared first in \cite{Ber2}, and we make heavy use of them. Roughly speaking, formulas
say that right and left $D$-modules on $X$ are the coinvariants and invariants of the category $\QC{X}$ under the action infinitesimal groupoid.

Pushforward and pullback
of $D$-modules are relatively easy to define in our setting, when we view the category of $D$-modules
as a colimit/limit of categories. However, we say nothing about the full six-functor formalism
or of holnomicity in this paper.

Key to our aim of showing that our theory of $D$-modules agrees with other classical definitions
is the proof of Kashiwara's Equivalence. Roughly speaking, our strategy here is to show that
both the category of $D$-modules on a closed subscheme $Z$ of $X$ and the category of $D$-modules
on $X$ supported on $Z$ are monadic over $\QC{Z}$. Then, we conclude by showing the two monads agree.

We also remark that almost all of the algebraic geometry input for setting up the $D$-module theory
appeared already in \cite{previous}. This paper uses
mostly categorical techniques.
We have avoided the use of stacks and formal schemes in this paper, relying more on categories such as $\Gamma_Z(\QC{X})$ (the category
of quasicoherent sheaves on $X$ supported on $Z$). We expect that with minor modifications, it should be possible
to develop $D$-modules in the analytic setting using similar categorical techniques.

\textbf{Inspirations}: This paper originally stemed from a desire to understand the relationship between $D$-modules
and the $E_2$-structure on Hochschild cohomology. We were heavily influenced by the papers of Beraldo, \cite{Ber1} and \cite{Ber2}, which
study the monoidal category $\mathbb{H}$, a close cousin of $\Gamma_{\Delta}(\QC{X \times X})$. Additionally, the various works of
Neeman and his coauthors persuaded us that the category $\Gamma_{\Delta}(\QC{X \times X})$ was indeed worthwhile to study.

\textbf{Speculations}: Let us offer a few speculative remarks on why the categorified ring of differential operators
\[\mathcal{D}_X := \Gamma_{\Delta}(\QC{X \times_S X}\]
should be helpful for defining the standard $D$-modules. I believe this stems from the fact that our schemes are $1$-affine,
combined with the fact that Grothendieck duality one categorical level higher is trivial for qcqs schemes. That is,
the categorified dualizing complex for $X$ is just $\QC{X}$. The reason is that the dualizing complex measure the difference
between the left and right adjoints of the pushforward functor (at least when the scheme is proper).
At the categorified level, the categorified pullback
\[2\mathrm{-}f^* : 2\mathrm{-}\QC{Y} \to 2\mathrm{-}\QC{X}\]
for a map of qcqs schemes $f: X \to Y$, is both left and right adjoint to the pushforward
\[2\mathrm{-}f^* : 2\mathrm{-}\QC{X} \to 2\mathrm{-}\QC{Y}\]
where $2-\QC{X}$ denotes the $2$-category of quasicoherent sheaves of categories over $U \mapsto \QC{U}$ (which is just
$\QC{X}$-module categories when $X$ is $1$-affine). Hence, the categorified dualizing complex must be trivial ($X$ is $1$-proper
in some sense).

Therefore, I believe that categorified $D$-modules are a simpler theory, and hence, decategorifying the
categorified theory of $D$-modules provides a relatively simple approach to defining $D$-modules.

\textbf{Outline}: Section 1 and 2 constructs the category of $D$-modules,
both left and right, and shows the left-right switch isomorphism. We also reprove some formulas for the category
of $D$-modules first seen in \cite{Ber2}, namely (\ref{catisoD}) and (\ref{catisoD2}) above.
Section 3 constructs the pushforward and pullback
functors of $D$-modules. Section 4 shows Kashiwara's equivalence. Section 5 discusses descent results
of the category of $D$-modules. Section 6 compares the category of $D$-modules with crystals. Section
7 discusses the relationship of our theory and the results of \cite{Cusp}. Section 8 discusses a decategorification
of (\ref{catisoD}) in the smooth case.

\textbf{Conventions}: All categories, unless stated otherwise will be $(\infty, 1)$-categories. A $2$-category
will refer to a $(\infty, 2)$-category. All functors, such as $\Hom$, $\otimes$, $\colim$, and $\lim$ will be
fully derived/done at the $\infty$-categorical level unless stated otherwise. A stable category will refer to
a stable $\infty$-category. All modules/quasicoherent sheaves will also be assumed to be fully derived.
We will aim to follow the terminology of Lurie in \cite{HTT}, \cite{HA}, and \cite{SAG}.

\textbf{Acknowledgements}: I heartily thank my advisor, Bhargav Bhatt, for countless discussions, suggestions, 
and insights whose effects permeate this paper. He also suggested showing that deriving the classical definition
of differential operators recovers the ``correct'' theory of $D$-modules on singular varieties.
I would also like to thank German Stefanich for patiently explaining
to me many ideas of the papers of Dario Beraldo, \cite{Ber1} and \cite{Ber2}. Additional thanks goes to
Shubhodip Mondal, Gleb Terentiuk, and Sridhar Venkatesh for many helpful discussions. Finally, I would like to thank various users of the Algebraic Topology Discord Channel for helpful discussions.

\section{The Category of $\opposite{D_X}$-modules} \label{Dmod}
In this section we define the category of $\opposite{D_X}$-modules and identify it with
the category of modules over a monad on $\QC{X}$ corresponding to the ``opposite'' of the sheaf $D_{X/S}$
defined in \ref{bigdef}.
Our approach is somewhat similar to the approach taken in section $5$ of
the paper \textit{D-modules and Crystals} \cite{CryD} by Gaitsgory and Rozenblyum.
However their starting point is de Rham stack and the completion of $X \times X$ at
the diagonal (part of what they call the infinitesimal groupoid) is defined in terms
of the de Rham stack.
In our approach we do the reverse. We view their approach as more stack-theoretic
and ours as more category-theoretic. This justifies our choice to give a self-contained
presentation of an arguably well-known theory. From a pedagogical perspectively, our presentation also
has the benefit of not relying on the theory of stacks and ind-coherent sheaves. However, we do have
to limit ourselves to the finite tor-amplitude situation (roughly the eventually coconnective situation
in the language of \cite{CryD}).

Recall that we defined $D_{X/S}$ as
\[D_{X/S} := \tilde{\pi}_1^{\times}\mathcal{O}_X \in \Gamma_{\Delta}(\QC{X \times_S X})\]
in Definition \ref{bigdef}, viewing it as an element of
\[\QC{X \times_S X} \cong \Hom_{\Mod{\QC{S}}^L}(\QC{X},\QC{X})\]
Here the tilde refers to fact that we apply the projection functor $\Gamma_{\Delta}$ after the $\pi_1^{\times}$.
In general, we use tilde to denote modification of functors which are related to the unmodified version
by the functors 
\[\begin{tikzcd}
    \QC{X'} \arrow[r, shift left=1ex, "\Gamma_{Z'}"{name=G}] & \Gamma_{Z'}(\QC{X'})\arrow[l, shift left=.5ex, "i_{Z'}"{name=F}]
    \arrow[phantom, from=F, to=G, , "\scriptscriptstyle\boldsymbol{\top}"].
  \end{tikzcd} \]
relating the category of quasicoherent sheaves supported on a closed subscheme $Z'$ of $X'$
with the entire category of quasicoherent sheaves on $X'$.

We can identify $D_{X/S}$ with a colimit-preserving $\QC{S}$-linear endofunctor of $\QC{X}$.
We would like to elevate its status to a monad on $\QC{X}$. In fact, for technical reasons, it will
be more convenient to begin with its ``opposite''. The ``opposite'' of $D_{X/S}$ corresponds to the endofunctor
which is left-right dual (in the sense of Section 3 of \cite{previous}) to the endofunctor of $D_{X/S}$.
As an element of $\Gamma_{\Delta}(\QC{X \times X})$,
$\opposite{D_{X/S}}$ is the image of $D_{X/S}$ under the automorphism of $\Gamma_{\Delta}(\QC{X \times X})$, which switches the $X$'s.
To be precise, we define
\[\opposite{D_{X/S}} := \tilde{\pi}_2^{\times}\mathcal{O}_X \in \Gamma_{\Delta}(\QC{X \times_S X})\]
The corresponding endofunctor to $\opposite{D_{X/S}}$ is $\tilde{\pi}_{1,*}\tilde{\pi}_2^{\times}$.
It is this endofunctor that we wish to make into a monad.

To do this, we will explicitly write down a category which we propose is the category of modules over
our desired monad. In Remark \ref{motivate}, we describe how to arrive at the following description, though expressing
the category of $D$-modules in this form is by no means new to us (see \cite{CryD} for instance).
Roughly speaking, this category is the colimit of the simplicial diagram 
\[\ldots \Gamma_{\Delta}(\QC{X \times X \times X}) \rightthreearrow \Gamma_{\Delta}(\QC{X \times X}) \righttwoarrow \QC{X}\]
where the transition maps are (tilde of) quasicoherent pushforward maps. For example, the two maps
\[\Gamma_{\Delta}(\QC{X \times X}) \righttwoarrow \QC{X}\]
are simply $\tilde{\pi}_{1,*}$ and $\tilde{\pi}_{2,*}$.

To be more precise, consider the simplex category $\bm{\Delta}$ consisting of objects
$\{[n]\}_{n \ge 0}$ where $[n]=\{0,\ldots,n\}$, and morphisms order-preserving (preserving $\ge$) maps between them.
We can define a functor \[\opposite{\Delta} \to \Mod{\QC{S}}^L\]
by sending 
\[[n] \mapsto \Gamma_{\Delta}\QC{X^{n+1}}\]
and an order preserving map $[n] \to [m]$ to the functor
\[\tilde{g}_*: \Gamma_{\Delta}(\QC{X^{m+1}}) \to \Gamma_{\Delta}(\QC{X^{n+1}})\]
where $g: X^{m+1} \to X^{n+1}$ is defined in the obvious way from the map $[n] \to [m]$.
The category which we propose is the category of right $D_X$ modules is then the colimit of this functor, for which we write
\begin{equation*}
  \colim_{\opposite{\bm{\Delta}}}(\QC{X^{n+1}},*)
\end{equation*}

Let us denote by $\bm{\Delta_s}$ the subcategory of $\bm{\Delta}$ where the morphisms are required to be injective.
By \cite{HTT} 6.5.3.7, the category $\opposite{\bm{\Delta_s}}$ is cofinal in $\opposite{\bm{\Delta}}$, and hence our colimit above
can be computed over $\opposite{\bm{\Delta_s}}$ instead, as
\begin{equation}
  \colim_{\opposite{\bm{\Delta_s}}}(\QC{X^{n+1}},*)
\end{equation}
The advantage of using $\bm{\Delta_s}$ is that for any injective morphism $[n] \to [m]$, the transition functor
\[\tilde{g}_*: \Gamma_{\Delta}(\QC{X^{m+1}}) \to \Gamma_{\Delta}(\QC{X^{n+1}})\]
described above is compact object preserving, by a mild generalization of Theorem 1.3 in \cite{previous}. The proof
is identical so we do not repeat it here. Intuitively, when taking the colimit over $\opposite{\bm{\Delta}}$, one encounters
degeneracy maps of simplices which induce functors such as
\[\tilde{\delta}_* : \QC{X} \to \Gamma_{\Delta}(\QC{X \times X})\]
which are not compact object preserving if $X$ is not smooth. This problem disappears when we use $\bm{\Delta}_s$.

Let us denote by
\[F_{\opposite{D_X}} : \QC{X} \to \colim_{\bm{\Delta_s}}(\Gamma_{\Delta}(\QC{X^{n+1}}),*) \]
the functor associated with the object $[0]$ in $\bm{\Delta_s}$, coming from the universal property
of the colimit. Denote by $G_{D_X}$ its right adjoint. We record here the following description of $G_{D_X}$.
Recall that the underlying category of a colimit in $Pr_{St}^L$ can also be written as a limit
in $Pr_{St}^R$, with the transition functors the right adjoints. This fact is due to Lurie \cite{HTT}, however we find Lemma 1.3.3
in \cite{DGCat} the most convenient reference. The essence is that adjunction provides an anti-equivalence of categories
between $Pr_{St}^L$ and $Pr_{St}^R$. 
With this in mind, $G_{\opposite{D_X}}$ can be written as the projection map
\[G_{\opposite{D_X}} : \lim_{\bm{\Delta_s}}(\Gamma_{\Delta}(\QC{X^{n+1}}),\times) \to \QC{X}\]
where the transition maps are tilde of upper cross functors (the right adjoint of tilde of lower star)
and the limit is taken in $Pr_{St}^R$. We remind the reader that if we are only interested in the
underlying category of the limit we can also take the limit in $\widehat{Cat_{\infty}}$. This functor is
also $\QC{S}$-linear (see Theorem A.6 of \cite{previous}).

Our aim for the rest of the section is to show that adjunction above is monadic, with the monad given by
\[G_{\opposite{D_X}}F_{\opposite{D_X}} \cong \tilde{\pi}_{1,*}\tilde{\pi}_2^{\times}\]
However, we will need a few preliminary results

\begin{lemma} \label{tensorlemma}
  For $m,n \ge 0$, there is a canonical isomorphism
  \[\Gamma_{\Delta}(\QC{X^{m+n+1}}) \cong \Gamma_{\Delta}(\QC{X^{m+1}}) \otimes_{\QC{X}} \Gamma_{\Delta}(\QC{X^{n+1}})\]
  where $\QC{X}$ acts on the right most copy of $X$ in $\Gamma_{\Delta}(\QC{X^{m+1}})$ 
  and the left most copy of $X$ in $\Gamma_{\Delta}(\QC{X^{n+1}})$ via tilde $*$-pullback.
\end{lemma}
\begin{proof}
  Because tensor products preserves split-exact sequences (see Appendix B of \cite{previous} for the definition of a split-exact sequence),
  both sides are full subcategories of
  \[\QC{X^{m+n+1}} \cong \QC{X^{m+1}} \otimes_{\QC{X}} \QC{X^{n+1}}\]
  It suffices to show they have the same objects.
  Let us denote by $U$ the complement of the diagonal in $X^{m+n+1}$.
  The category $\Gamma_{\Delta}(\QC{X^{m+n+1}})$ can then be characterized as the subcategory of $\QC{X^{m+n+1}}$ which vanish
  when restricted to $U$.

  Now let $V$ be the complement of the diagonal in $X^{m+1}$ and $W$ the complement of the diagonal in $X^{n+1}$. Then, we can express $U$ as a union
  \[U = V \times_X X^{n+1} \cup X^{m+1} \times_X W\]
  Therefore, vanishing on $U$ is equivalent to vanishing on $V \times_X X^{n+1}$ and $X^{m+1} \times_X W$.

  It is then clear that everything in
  \[\Gamma_{\Delta}(\QC{X^{m+1}}) \otimes_{\QC{X}} \Gamma_{\Delta}(\QC{X^{n+1}})\]
  vanishes on $U$. For the reverse, suppose a quasicoherent sheaf $\mathcal{F}$ vanishes on $U$.
  It then lives inside $\Gamma_{\Delta}(\QC{X^{m+1}}) \otimes_{\QC{X}} \QC{X^{n+1}}$ because
  it vanishes on $V \times_X X^{n+1}$. Then, because it also vanishes on $X^{m+1} \times_X W$, it is then inside the kernel of the map
  \[\Gamma_{\Delta}(\QC{X^{m+1}}) \otimes_{\QC{X}} \QC{X^{n+1}} \to \Gamma_{\Delta}(\QC{X^{m+1}}) \otimes_{\QC{X}} \QC{W}\]
  Because tensor product of stable categories preserve split-exact sequences, we see that $\mathcal{F}$
  is inside
  \[\Gamma_{\Delta}(\QC{X^{m+1}}) \otimes_{\QC{X}} \Gamma_{\Delta}(\QC{X^{n+1}})\]
\end{proof}
\begin{remark}
  The previous lemma leads to an interesting observation. The simplicial diagram
  \[[n] \mapsto \Gamma_{\Delta}(\QC{X^{n+1}})\]
  roughly specifies the data of category internal to $\Mod{\QC{S}}^L$ on the object $\QC{X}$,
  relative to the tensor product of categories. This internal category is the categorical analogue
  of the infinitesimal groupoid on $X$.
\end{remark}

We move to the second preliminary result.
Recall that we can equip $\Gamma_{\Delta}(\QC{X \times X})$ with the convolution monoidal structure (Definition 1.11 of \cite{previous}).
Under this monoidal structure, 
$\QC{X}$ is naturally a left $\Gamma_{\Delta}(\QC{X \times X})$ module. We will give a resolution of $\QC{X}$
as a $\Gamma_{\Delta}(\QC{X \times X})$ module. We exhibit this resolution as an augmented simplicial diagram
\begin{equation}
  \ldots \Gamma_{\Delta}(\QC{X \times X \times X}) \righttwoarrow \Gamma_{\Delta}(\QC{X \times X}) \to \QC{X}
\end{equation}
The augmentation map is \[\tilde{\pi}_{1,*} : \Gamma_{\Delta}(\QC{X \times X}) \to \QC{X}\]
The two maps
\[\Gamma_{\Delta}(\QC{X \times X \times X}) \righttwoarrow \Gamma_{\Delta}(\QC{X \times X})\]
are $\tilde{\pi}_{1,2,*}$ and $\tilde{\pi}_{1,3,*}$. More generally,
all the transition maps preserve the left most copy of $X$. We omit writing down the complete specification
of this simplicial diagram and trust that the reader is able to do so if they wish.
Importantly, the action of $\Gamma_{\Delta}(\QC{X \times X})$ is always
on the left most copy of $X$ which is preserved. The following proposition shows that
this is indeed a resolution, i.e. that the geometric realization of the simplicial diagram recovers $\QC{X}$.

\begin{proposition} \label{resolve}
  There is a natural resolution of $\QC{X}$ as a left $\Gamma_{\Delta}(\QC{X \times X})$-module category given by
  \begin{equation}
    \ldots \Gamma_{\Delta}(\QC{X \times X \times X}) \righttwoarrow \Gamma_{\Delta}(\QC{X \times X}) \to \QC{X}
  \end{equation}
  where the maps are specified above.
\end{proposition}
\begin{proof}
  We apply Lemma 6.1.3.17 from \cite{HTT}. The augmented simplicial diagram above arises from a simplicial object
  \[\ldots \Gamma_{\Delta}(\QC{X \times X \times X}) \rightthreearrow \Gamma_{\Delta}(\QC{X \times X}) \righttwoarrow \QC{X}\]
  by forgetting all the morphisms which do not preserve the left most copy of $X$. 
  Therefore it is a colimit diagram in $\Mod{\QC{S}}^L$. Because the forgetful functor
  from $\Mod{\Gamma_{\Delta}(\QC{X \times X})}^L$ to $\Mod{\QC{S}}^L$ reflects colimits (because it preserves colimits
  by Corollary 4.2.3.7 of \cite{HA} and is conservative), it is also a colimit diagram in $\Mod{\Gamma_{\Delta}(\QC{X \times X})}^L$.
\end{proof}

The above proposition has an important corollary.
\begin{corollary} \label{keycoro}
  \begin{equation} \label{keyiso}
  \colim_{\bm{\opposite{\Delta}}}(\Gamma_{\Delta}(\QC{X^{n+1}}),*) \cong \QC{X} \otimes_{\Gamma_{\Delta}(\QC{X \times X})} \QC{X}
  \end{equation}
\end{corollary}
\begin{proof}
  Using Proposition \ref{resolve}, we can write the right hand side as
  \[\colim_{\opposite{\bm{\Delta_s}}}(\QC{X} \otimes_{\Gamma_{\Delta}(\QC{X \times X})}\Gamma_{\Delta}(\QC{X^{n+2}})\]
  Using Lemma \ref{tensorlemma}, we can write $\Gamma_{\Delta}(\QC{X^{n+2}})$ as
  \[\Gamma_{\Delta}(\QC{X \times X}) \otimes_{\QC{X}} \Gamma_{\Delta}(\QC{X^{n+1}})\]
  Therefore, the right hand side is isomorphic to 
  \[\colim_{\opposite{\bm{\Delta_s}}}(\Gamma_{\Delta}(\QC{X^{n+1}}),*)\]
  as desired.
\end{proof}
\begin{remark}
  This expression for the category of right $D$-modules appears as Equation (4.5) in the proof of Proposition 4.2.5 in \cite{Ber2}.
\end{remark}

The adjunction between $F_{\opposite{D_X}}$ and $G_{\opposite{D_X}}$ can be described in terms of the isomorphism
above. Because of the isomorphism
\[\QC{X} \cong \QC{X} \otimes_{\Gamma_{\Delta}(\QC{X \times X})} \Gamma_{\Delta}(\QC{X \times X})\]
The map
\[\tilde{\pi}_{1,*} : \Gamma_{\Delta}(\QC{X \times X}) \to \QC{X}\]
induces a functor
\[\id \otimes \tilde{\pi}_{1,*} : \QC{X} \to \QC{X} \otimes_{\Gamma_{\Delta}(\QC{X \times X})} \QC{X}\]
Tracing through the proof of Corollary \ref{keycoro}, we see this agrees with functor $F_{\opposite{D_X}}$, after identifying the two sides of
Corollary \ref{keycoro}. Namely,
\begin{equation} \label{equa}
  F_{\opposite{D_X}} \cong \id \otimes \tilde{\pi}_{1,*}
\end{equation}
Now, the right adjoint of $\tilde{\pi}_{1,*}$,
\[\tilde{\pi}_1^{\times} : \QC{X} \to \Gamma_{\Delta}(\QC{X \times X})\]
is also $\Gamma_{\Delta}(\QC{X \times X})$ linear. In the absence of an enriched
adjunction theorem for modules over monoidal categories (which should exist but we did not find a version which applies to this case),
we can see this directly by the isomorphism
(see equation equation (2.3) and Proposition 2.6 of \cite{previous})
\[\tilde{\pi}_1^{\times}(\mathcal{F}) \cong \Gamma_{\Delta}(\mathcal{F} \boxtimes \omega_X)\]
Hence, we can construct the functor
\[\id \otimes \tilde{\pi}_1^{\times} : \QC{X} \otimes_{\Gamma_{\Delta}(\QC{X \times X})} \QC{X} \to \QC{X}\]
Using the unit and counit maps of the adjunction $\tilde{\pi}_{1,*} \dashv \tilde{\pi}_1^{\times}$, we can see that
our $\id \otimes \tilde{\pi}_1^{\times}$ is right adjoint to $F_{\opposite{D_X}}$ and hence
\begin{equation} \label{equaa}
  G_{\opposite{D_X}} \cong \id \otimes \tilde{\pi}_1^{\times}
\end{equation}

By examination, or by the involution on
\[\colim_{\opposite{\bm{\Delta_s}}}(\Gamma_{\Delta}(\QC{X^{n+1}}),*)\]
which reverse the order of the $X$'s in $X^{n+1}$ (for all $n$), we can also arrive at the isomorphism (\ref{keyiso})
through a resolution of the left copy of $\QC{X}$ as a right $\Gamma_{\Delta}(\QC{X \times X})$ module
(analogously to Proposition \ref{resolve})
By arriving at the isomorphism this way, we can also express $F_{\opposite{D_X}}$ as
\begin{equation}
  F_{\opposite{D_X}} \cong \tilde{\pi}_{2,*} \otimes \id
\end{equation}
This expression for $F_{\opposite{D_X}}$ implies
\begin{equation} \label{equb}
  G_{\opposite{D_X}} \cong \tilde{\pi}_2^{\times} \otimes \id 
\end{equation}

Finally, we can deliver on our promise
\begin{theorem} \label{bigtheorem}
  The adjunction $F_{\opposite{D_X}} \dashv G_{\opposite{D_X}}$ is monadic and
  \begin{equation}
    G_{\opposite{D_X}}F_{\opposite{D_X}} \cong \tilde{\pi}_{1,*}\tilde{\pi}_2^{\times}
  \end{equation}
\end{theorem}
\begin{proof}
  By Lurie-Barr-Beck (Theorem 4.7.3.5 in \cite{HA}), to show the adjunction is monadic
  it is enough to show that $G_{\opposite{D_X}}$ is conservative and colimit-preserving. Because
  all the transition maps used to construct the cosimplicial limit
  \[\lim_{\bm{\Delta_s}}(\Gamma_{\Delta}(\QC{X^{n+1}}),\times)\]
  are colimit-preserving, $G_{\opposite{D_X}}$ is automatically colimit-preserving. To show that $G_{\opposite{D_X}}$ is conservative,
  we need to show that an object 
  \[\mathcal{F} \in \lim_{\bm{\Delta_s}}(\Gamma_{\Delta}(\QC{X^{n+1}}),\times)\]
  is zero if the projection of $\mathcal{F}$ to $\QC{X}$ is zero. For this, it is enough
  to show that the projection to $\Gamma_{\Delta}(\QC{X^{n+1}})$ is zero for any $n$. But follows from the fact that
  $[0]$ is weakly initial in $\bm{\Delta_s}$.
  
  For the second part of the theorem, let us apply 
  Theorem A.1 of \cite{previous} to $\mathscr{X} := \QC{X}$, $\mathscr{Y} := \QC{X}$,
  and $\mathscr{V} := \Gamma_{\Delta}(\QC{X \times X})$, with functors $\tilde{\pi}_2^{\times} : \mathscr{X} \to \mathscr{V}$ and
  $\tilde{\pi}_{1,*} : \mathscr{V} \to \mathscr{Y}$. Then, we have the commutative diagram
  \[
    \begin{tikzcd}
      \QC{X}  \arrow{r}{1 \otimes \tilde{\pi}_{1,*}} \arrow{dd}{\tilde{\pi}_2^{\times}} &
      \QC{X} \otimes_{\Gamma_{\Delta}(\QC{X \times X})} \QC{X} \arrow{dd}{\tilde{\pi}_2^{\times} \otimes 1} \\ \\
      \Gamma_{\Delta}(\QC{X \times X}) \arrow{r}{\tilde{\pi}_{1,*}} & \QC{X} \\
    \end{tikzcd}
  \]
  Now the theorem follows from expressions (\ref{equa}) and (\ref{equb}).
\end{proof}
\begin{remark}
  We also have
  \begin{equation}
    G_{\opposite{D_X}}F_{\opposite{D_X}} \cong \tilde{\pi}_{2,*}\tilde{\pi}_1^{\times}
  \end{equation}
  Because of the isomorphism
  \[\tilde{\pi}_{2,*}\tilde{\pi}_1^{\times} \cong \tilde{\pi}_{1,*}\tilde{\pi}_2^{\times}\]
  which can be simply explained by observing there is a symmetry which switches the order of
  the two $X$'s in $\Gamma_{\Delta}(\QC{X \times X})$.
\end{remark}

\begin{remark}
  The functor $F_{\opposite{D_X}} : \QC{X} \to \QC{X} \otimes_{\Gamma_{\Delta}(\QC{X \times X})} \QC{X}$ can also be arrived at
  via the monoidal functor
  \[\tilde{\delta_*} : \QC{X} \to \Gamma_{\Delta}(\QC{X \times X})\]
  which by functoriality induces a functor
  \[\QC{X} \otimes_{\QC{X}} \QC{X} \to \QC{X} \otimes_{\Gamma_{\Delta}(\QC{X \times X})} \QC{X}\]
  as desired. The fact that this agrees with the prior definitions can be checked using one of the resolutions above.
\end{remark}

\begin{corollary}
  $\opposite{D_{X/S}}$ can be made into an algebra with respect to the convolution monoidal product on $\Gamma_{\Delta}(\QC{X \times X})$.
\end{corollary}
\begin{proof}
  The endofunctor corresponding to $\opposite{D_{X/S}}$ is $\tilde{\pi}_{1,*}\tilde{\pi}_2^{\times}$, so the result follows from the theorem.
\end{proof}

From now on, when we refer to $\opposite{D_{X/S}}$ as an algebra, we will implicitly refer to this algebra structure.
\begin{theodef} \label{rightdref}
  The category of $\opposite{D_{X/S}}$-modules is defined to be the category of algebras over the monad $\tilde{\pi}_{1,*}\tilde{\pi}_2^{\times}$
  (see Theorem \ref{bigtheorem} for why this endofunctor is a monad).
  We have the isomorphisms 
  \[\Mod{\opposite{D_{X/S}}} \cong \colim_{\opposite{\bm{\Delta_s}}}(\Gamma_{\Delta}(\QC{X^{n+1}}),*)\]
  \[\Mod{\opposite{D_{X/S}}} \cong \lim_{\bm{\Delta_s}}(\Gamma_{\Delta}(\QC{X^{n+1}}),\times)\]
  Additionally, we have the isomorphism
  \[\Mod{\opposite{D_{X/S}}} \cong \QC{X} \otimes_{\Gamma_{\Delta}(\QC{X \times X})} \QC{X}\]
  Also,
  $\Mod{\opposite{D_{X/S}}}$ satisfies étale descent with respect to $X$ and fpqc descent with respect to $S$.
\end{theodef}
\begin{proof}
  All but the last sentence follow directly from Theorem \ref{bigtheorem} and Corollary \ref{keycoro}. For the last part,
  use the isomorphism
  \[\Mod{\opposite{D_{X/S}}} \cong \lim_{\bm{\Delta_s}}(\Gamma_{\Delta}(\QC{X^{n+1}}),\times)\]
  each of the $\Gamma_{\Delta}(\QC{X^{n+1}})$ has étale descent with respect to $X$ (by a variant of Corollary 1.6 of \cite{previous})
  and all the transition maps base-change correctly (by a variant of Proposition 1.8 of \cite{previous}). The fpqc descent follow
  from fpqc descent of $\mathrm{QCoh}$ (Proposition 6.2.3.1 in \cite{SAG}) and the fact that we can write $\Gamma_Z(\QC{X})$ as the kernel in the split-exact sequence
  \[\Gamma_Z(\QC{X}) \to \QC{X} \to \QC{U}\]
  where $U$ is the complement of $Z$ in $X$.
\end{proof}

\begin{remark}
  In the affine situation when $X = \Spec{A}$ and $S = \Spec{k}$,
  \[\opposite{D_{X/S}} \cong \opposite{\Gamma_{\Delta}(\Hom_k(A,A))}\]
  is an algebra isomorphism. One can show this by carefully unwinding the definitions
  in the affine case and we leave the details to the reader.
\end{remark}

\begin{remark} \label{explicit1}
  It is also possible to prove Theorem \ref{bigtheorem} by giving an explicit description of $G_{\opposite{D_X}}$ as
  a functor
  \[G_{\opposite{D_X}} : \colim_{\opposite{\bm{\Delta_s}}}(\Gamma_{\Delta}(\QC{X^{n+1}}),*) \to \QC{X}\]
  To specify such a functor, it suffices to specify a collection of functors
  \[G_{\opposite{D_X}}^{(n)}: \Gamma_{\Delta}(\QC{X^{n+1}},*) \to \QC{X}\]
  together with compatibility isomorphisms. We call the process of constructing $G_{\opposite{D_X}}$ from the
  $G_{\opposite{D_X}}^{(n)}$'s \textbf{assembly}. By equation (\ref{equaa}),
  we can write $G_{\opposite{D_X}}$ as the map induced (via colimit over $\opposite{\bm{\Delta_s}}$) by the following
  map of simplicial diagrams
  \[
    \begin{tikzcd}
      \ldots \Gamma_{\Delta}(\QC{X \times X \times X})  \arrow[r, yshift=1.2ex] \arrow[r] \arrow[r, yshift=-1.2ex] \arrow{d}{\tilde{\pi}_{1,2,3}^{\times}}
      & \Gamma_{\Delta}(\QC{X \times X}) \arrow[r, yshift=0.7ex] \arrow[r, yshift=-0.7ex] \arrow{d}{\tilde{\pi}_{1,2}^{\times}} & \QC{X} \arrow{d}{\tilde{\pi}_1^{\times}}\\
      \ldots \Gamma_{\Delta}(\QC{X \times X \times X \times X}) \arrow[r, yshift=1.2ex] \arrow[r] \arrow[r, yshift=-1.2ex]
      & \Gamma_{\Delta}(\QC{X \times X \times X}) \arrow[r, yshift=0.7ex] \arrow[r, yshift=-0.7ex] & \Gamma_{\Delta}(\QC{X \times X}) \\
    \end{tikzcd}
  \]
  Therefore, we can compute 
  \[G_{\opposite{D_X}}^{(n)} \cong \tilde{\pi}_{n+2,*}\tilde{\pi}^{\times}_{\widehat{n+2}}\]
  where 
  \[\tilde{\pi}^{\times}_{\widehat{n+2}} : \Gamma_{\Delta}(\QC{X^{n+1}}) \to \Gamma_{\Delta}(\QC{X^{n+2}})\]
  is defined as tilde upper cross for the projection map $\pi_{\widehat{n+2}}$ to all but the last component of the product.
\end{remark}
The above remark is generalized by
\begin{theorem} \label{explicit3}
  The identity functor
  \[\colim_{\opposite{\bm{\Delta_s}}}(\Gamma_{\Delta}(\QC{X^{n+1}}),*) \to \lim_{\bm{\Delta_s}}(\Gamma_{\Delta}(\QC{X^{m+1}},\times))\]
  is assembled from the functors
  \[\tilde{\pi}_{\{n+2,\ldots,n+m+2\},*}\tilde{\pi}_{\{1,\ldots,n+1\}}^{\times} : \Gamma_{\Delta}(\QC{X^{n+1}}) \to \Gamma_{\Delta}(\QC{X^{m+1}})\]
  where 
  \[\tilde{\pi}_{\{n+2,\ldots,n+m+2\},*} : \Gamma_{\Delta}(\QC{X^{n+m+2}}) \to \Gamma_{\Delta}(\QC{X^{m+1}})\]
  \[\tilde{\pi}_{\{1,\ldots,n+1\}}^{\times} : \Gamma_{\Delta}(\QC{X^{n+1}}) \to \Gamma_{\Delta}(\QC{X^{n+m+2}})\]
  with the obvious transition functors. Therefore, $F_{\opposite{D_X}}$ is assembled from the functors
  \[F_{\opposite{D_X}}^{(n)} \cong \tilde{\pi}_{\hat{1},*}\tilde{\pi}^{\times}_1\]
\end{theorem}
\begin{proof}
  Analogous to equation (\ref{equa}), the inclusion functor
  \[i_m: \Gamma_{\Delta}(\QC{X^{m+1}} \to \colim_{\opposite{\bm{\Delta_s}}}(\Gamma_{\Delta}(\QC{X^{m+1}}),*) \]
  can be written also as
  \[i_m \cong \id \otimes \tilde{\pi}_{1,*} : \QC{X} \otimes_{\Gamma_{\Delta}(\QC{X \times X})} \Gamma_{\Delta}(\QC{X^{m+2}}) \to \QC{X} \otimes_{\Gamma_{\Delta}(\QC{X \times X})} \QC{X}\]
  Hence, its right adjoint is
  \[\id \otimes \tilde{\pi}_1^{\times} : \QC{X} \otimes_{\Gamma_{\Delta}(\QC{X \times X})} \QC{X} \to \QC{X} \otimes_{\Gamma_{\Delta}(\QC{X \times X})} \Gamma_{\Delta}(\QC{X^{m+2}})\]
  Now, we can resolve the left $\QC{X}$ in the tensor
  \[\QC{X} \otimes_{\Gamma_{\Delta}(\QC{X \times X})} \QC{X}\]
  as a right $\Gamma_{\Delta}(\QC{X \times X})$ modules (analogously to Proposition \ref{resolve}). 
  Using this resolution the right adjoint of $i_m$ can be written as the assembly of
  \[(\id \otimes \tilde{\pi}_1^{\times}) \circ (\tilde{\pi}_{n+1,*} \otimes \id)\]
  from
  \[\Gamma_{\Delta}(\QC{X^{n+2}}) \otimes_{\Gamma_{\Delta}(\QC{X \times X})} \QC{X}\]
  to
  \[\QC{X} \otimes_{\Gamma_{\Delta}(\QC{X \times X})}\Gamma_{\Delta}(\QC{X^{m+2}})\]
  By the functoriality of the tensor product, this is also the same as
  \[\tilde{\pi}_{\{n+2,\ldots,n+m+2\},*}\tilde{\pi}_{\{1,\ldots,n+1\}}^{\times} : \Gamma_{\Delta}(\QC{X^{n+1}}) \to \Gamma_{\Delta}(\QC{X^{m+1}})\]
  where the transition isomorphisms are obvious. As taking the right adjoint of $i_m$ yields
  also the identity functor in the theorem composed with the projection to the $m$-th component of the limit, we recover the theorem.
\end{proof}

\begin{remark} \label{motivate}
  Theorem/Definition \ref{rightdref} shows an equivalence between right
  $D$-modules and costratifications (the name commonly given to the category on the right)
  \[\Mod{\opposite{D_X}} \cong \lim_{\bm{\Delta}}(\Gamma_{\Delta}(\QC{X^{n+1}}),\times)\]
  We describe how to arrive at this equivalence naturally.
  Suppose $M$ is a right $D_X$ module, so that there is a map
  \[\tilde{\pi}_{2,*}(\pi_1^{\times}(M)) \to M\]
  By adjunction this is the same as a map
  \[\phi : \tilde{\pi}_1^{\times}M \to \tilde{\pi}_2^{\times}M\]
  which can also be written as
  \[\phi : \Gamma_{\Delta}(M \boxtimes \omega_X) \to \Gamma_{\Delta}(\omega_X \boxtimes M)\]
  Being a right $D_X$ module includes also higher compatibilities. These include things such as the commutativity of the following diagram
  \[
    \begin{tikzcd}
      \Gamma_{\Delta}(\omega_X \boxtimes \omega_X \boxtimes M) \arrow{rr}{\tilde{\pi}_{2,3}^{\times}(\phi)} \arrow{dd} &&
      \Gamma_{\Delta}(\omega_X \boxtimes M \boxtimes \omega_X) \arrow{dd}{\tilde{\pi}_{1,2}^{\times}(\phi)} \\ \\
      \Gamma_{\Delta}(\omega_X \boxtimes \omega_X \boxtimes M) \arrow{rr}{\tilde{\pi}_{1,3}^{\times}(\phi)} &&
      \Gamma_{\Delta}(M \boxtimes \omega_X \boxtimes \omega_X)\\
    \end{tikzcd}
  \]
  where the left unlabeled map is the identity. All the maps above are also required to be isomorphisms
  upon cross pullback to $\QC{X}$ along the diagonal map.
  Because upper cross pullback along the diagonal is conservative (for quasicoherent sheaves supported on the diagonal)
  all the above maps are isomorphisms. The above discussion explains the first three terms of the limit
  \[\QC{X} \righttwoarrow \Gamma_{\Delta}(\QC{X \times X}) \rightthreearrow \Gamma_{\Delta}(\QC{X \times X \times X}) \ldots \]
  which we showed was equivalent to the category of right $D$ modules.
\end{remark}


\section{$D_X$-Modules and Left-Right Switch}
In this section we discuss left $D$-modules and the isomorphism between the category of right $D$-modules
and left $D$-modules, which is called the left-right switch.

We have seen that $\opposite{D}_X$ is an algebra in $\Gamma_{\Delta}(\QC{X \times X})$. There is an involution on
$\Gamma_{\Delta}(\QC{X \times X})$ which switches the two copies of $X$. This is an auto-equivalence of the category
which reverses the monoidal structure. We endow $D_X$ with an algebra structureusing this involution.
As an algebra in $\Gamma_{\Delta}(\QC{X \times X})$ (with convolutional monoidal product),
$D_X$ also defines a monad on $\QC{X}$. We refer to modules over this monad as $D_X$-modules.

If we think of quasicoherent sheaves on $\Gamma_{\Delta}(\QC{X \times X})$ as endofunctors on $\QC{X}$,
then by Proposition 3.8 of \cite{previous}, we know that the involution above
is equivalent to the left-right duality on endofunctors of $\QC{X}$ (we remind the reader that $\QC{X}$ is naturally self-dual). Hence we see that the monad
$D_X$ is the left-right dual to $\opposite{D_X}$. We note that at the affine level, this is just the statement that
left-right duality takes an associative ring to its opposite.

By Corollary \ref{appendA}, we know that left-right duality switches
left and right $D_X$-modules. Namely,
\[(\Mod{\opposite{D_X}})^{\vee} \cong \Mod{D_X}\]
where $^{\vee}$ denotes duality in $\Mod{\QC{S}}^L$.
Moreover, by Corollary \ref{appendB} the adjunction
\[F_{\opposite{D_X}} \dashv G_{\opposite{D_X}}\]
becomes, under left-right duality, the adjunction
\[G_{D_X} \vdash F_{D_X}\]
We know from the last section that
\[\Mod{\opposite{D}_X} \cong \colim_{\opposite{\bm{\Delta_s}}}(\Gamma_{\Delta}(\QC{X^{n+1}}),*)\]
Applying the $2$-functor
\[\Hom_{\QC{S}}(\_,\QC{S})\]
(the $\Hom$ is taken inside $\Mod{\QC{S}}^L$) to the equation above, we get the isomorphism
\[\Mod{D_X} \cong \lim_{\bm{\Delta}_s}(\Gamma_{\Delta}(\QC{X^{n+1}}),*)\]
where the transition maps are tilde of quasicoherent pullbacks. This is because, by Proposition 3.10 of \cite{previous},
$\Gamma_{\Delta}(\QC{X^{n+1}})$ is (canonically) self-dual for all $n$ and therefore
\[\Hom_{\QC{S}}(\Gamma_{\Delta}(\QC{X^{n+1}}),\QC{S}) \cong \Gamma_{\Delta}(\QC{X^{n+1}})\]
Let's record our observations in

\begin{theorem} \label{bigtheorem2}
  $D_X$ is naturally an algebra in the monoidal category $\Gamma_{\Delta}(\QC{X \times_S X})$,
  corresponding to the monad $\tilde{\pi}_{1,\times}\tilde{\pi}_2^*$ and
  \begin{equation*}
    \begin{split}
      \Mod{D_X}
      &\cong \lim_{\bm{\Delta_s}}(\Gamma_{\Delta}(\QC{X^{n+1}}),*) \\
      &\cong \colim_{\opposite{\bm{\Delta_s}}}(\Gamma_{\Delta}(\QC{X^{n+1}}),\times) \\
      &\cong \Hom_{\Gamma_{\Delta}(\QC{X \times X})}(\QC{X},\QC{X}) \\
    \end{split}
  \end{equation*}
  Also $\Mod{D_X}$ satisfies étale descent with respect to $X$ and fpqc descent with respect to $S$.
\end{theorem}
\begin{proof}
  $\opposite{D_X}$ corresponds to the monad $\tilde{\pi}_{2,*}\tilde{\pi}_1^{\times}$
  by Theorem/Definition \ref{rightdref}. Therefore by left-right duality, $D_X$
  corresponds to the monad $\tilde{\pi}_{1,\times}\tilde{\pi}_2^*$
  \footnote{Tilde lower-cross means the left adjoint of tilde upper-star and is the left-right
    switch of tilde upper-cross, see Section 4 of \cite{previous}},
  the first isomorphism is already proven. The second isomorphism comes from the equivalence between
  colimits and limits in the form of Lemma 1.3.3 in \cite{DGCat}. For the third isomorphism, we give two proofs. 

  \textit{Proof 1. }
  By the resolution of $\QC{X}$ as a left $\Gamma_{\Delta}(\QC{X \times X})$ module category
  \[\ldots \Gamma_{\Delta}(\QC{X \times X \times X}) \righttwoarrow \Gamma_{\Delta}(\QC{X \times X}) \to \QC{X}\]
  (see Proposition \ref{resolve}), one can directly check that
  \[\Hom_{\Gamma_{\Delta}(\QC{X \times X})}(\QC{X},\QC{X}) \cong \lim_{\bm{\Delta_s}}(\Gamma_{\Delta}(\QC{X^{n+1}}),*) \]
  analogously to Corollary \ref{keycoro}.

  \textit{Proof 2.}
  \begin{equation*}
    \begin{split}
      \Mod{D_X}
      &\cong \Hom_{\QC{S}}(\Mod{\opposite{D_X}},\QC{S}) \\
      &\cong \Hom_{\QC{S}}(\QC{X} \otimes_{\Gamma_{\Delta}(\QC{X \times X})} \QC{X},\QC{S}) \\ 
      &\cong \Hom_{\Gamma_{\Delta}(\QC{X \times X}}(\QC{X}, \Hom_{\QC{S}}(\QC{X},\QC{S})) \\
      &\cong \Hom_{\Gamma_{\Delta}(\QC{X \times X})}(\QC{X},\QC{X}) \\
    \end{split}
  \end{equation*}

  The descent result is proven identically as in Theorem \ref{rightdref}.
\end{proof}

\begin{remark}
  The reader is encouraged to compare this result with Remark 1.8.4 in \cite{Ber2}
\end{remark}

\begin{remark} \label{explicit2}
  We can ask for an explicit description of the functor
  \[F_{D_X} : \QC{X} \to \lim_{\bm{\Delta_s}}(\Gamma_{\Delta}(\QC{X^{n+1}}),*)\]
  as a compatible system of functors
  \[F_{D_X}^{(n)}:\QC{X} \to \Gamma_{\Delta}(\QC{X^{n+1}},*)\]
  In fact, we have
  \[F_{D_X}^{(n)} \cong \tilde{\pi}_{\widehat{n+2},\times}\tilde{\pi}_{n+2}^* \]
  by left-right duality applied to Remark \ref{explicit1}. Here $\pi_{\widehat{n+2}}$ means
  projection to all but the $n+2$-th component.

  Additonally, using the isomorphism
  \[\Mod{D_X} \cong \Hom_{\Gamma_{\Delta}(\QC{X \times X})}(\QC{X},\QC{X}) \]
  we have the descriptions
  \[F_{D_X} \cong \Hom(\tilde{\pi}_1^{\times},\id)\]
  \[G_{D_X} \cong \Hom(\tilde{\pi}_{1,*},\id)\]
  which we can prove using left-right duality, the \textit{Proof 2.} above, and equations (\ref{equa}) and (\ref{equaa}).
\end{remark}

\begin{remark}
  The limit we gave for the category of $D_X$-modules
  \[\lim_{\bm{\Delta}}(\Gamma_{\Delta}(\QC{X^{n+1}}),*)\]
  can be seen to be the category of quasicoherent crystals on the stratifying site of $X$.
  If $X$ is a smooth variety over a field, we can use descent to see that this category
  is equivalent to the category of quasicoherent sheaves on the de Rham stack, via the
  Čech nerve of the map
  \[X \to X_{dR}\]
  In fact, $D_X$-modules are the same as quasicoherent sheaves on $X_{dR}$ in more generality.
  In characteristic zero this is Proposition 3.4.3 in \cite{CryD}. In our paper, 
  we will show this as a consequence Kashiwara's equivalence in Section \ref{cryssec}.
\end{remark}

\begin{remark}
  By expressing the category of $D_X$-modules as the cosimplicial limit above, we can see that $\Mod{D_X}$ is a symmetric monoidal category.
\end{remark}

Now it's time to discuss the left-right switch. We can construct an explicit functor $Q$
from $\Mod{\opposite{D_X}}$ to $\Mod{D_X}$ as follows. Recall
\[\Mod{\opposite{D_X}} \cong \QC{X} \otimes_{\Gamma_{\Delta}(\QC{X \times X})} \QC{X}\]
and 
\[\Mod{D_X} \cong \Hom_{\Gamma_{\Delta}(\QC{X \times X})}(\QC{X},\QC{X})\]
Therefore, the functor
\[\Gamma_{\Delta} \otimes \id : 
  \QC{X \times X} \otimes_{\Gamma_{\Delta}(\QC{X \times X})} \QC{X} \to \QC{X}\]
which can also be written as 
\[\Gamma_{\Delta} \otimes \id : \QC{X} \otimes \QC{X} \otimes_{\Gamma_{\Delta}(\QC{X \times X})} \QC{X} \to \QC{X}\]
 is $\Gamma_{\Delta}(\QC{X \times X})$-linear ($\Gamma_{\Delta}(\QC{X \times X})$ acts by convolution on the leftmost $\QC{X}$)
and therefore induces a functor
\[Q: \Mod{\opposite{D_X}} \to \Mod{D_X}\]

Since $Q$ is colimit-preserving, $Q$ can be represented by a $(D_X,\opposite{D_X})$-bimodule.
We can determine which bimodule it is by calculating $G_{D_X}QF_{\opposite{D_X}}$. By chasing through the definitions
and using equation (\ref{equa}) and its left-right dual, we can calculate
\[G_{D_X}QF_{\opposite{D_X}} \cong (\tilde{\pi}_2^* \otimes \id) \circ (\id \otimes \tilde{\pi}_{1,*})\]
This has domain
\[\QC{X} \cong \QC{X} \otimes_{\Gamma_{\Delta}(\QC{X \times X})} \Gamma_{\Delta}(\QC{X \times X})\]
and codomain
\[\QC{X} \cong \Gamma_{\Delta}(\QC{X \times X}) \otimes_{\Gamma_{\Delta}(\QC{X \times X})} \QC{X}\]
Hence by Theorem A.1 of \cite{previous}, 
\[G_{D_X}QF_{\opposite{D_X}} \cong \tilde{\pi}_{1,*}\tilde{\pi}_2^* \]
and the relevant $(D_X,\opposite{D}_X)$ bimodule is
\[\tilde{\pi}_1^*\mathcal{O}_X \cong \tilde{\pi}_2^*\mathcal{O}_X \cong \Gamma_{\Delta}(\mathcal{O}_{X \times X}) \cong \Gamma_{\Delta}(\mathcal{O}_X \boxtimes \mathcal{O}_X)\]
in $\Gamma_{\Delta}(\QC{X \times_S X})$.
\begin{remark}
  Strictly speaking, we have not defined what it means to be a $(D_X,\opposite{D_X})$-bimodule.
  $\Gamma_{\Delta}(\QC{X \times_S X})$ is a $\Gamma_{\Delta}(\QC{X \times_S X})$-bimodule category, therefore
  there is a monad obtained by combining the $D_X$ monad on the left with the $\opposite{D_X}$ monad on the right.
  A $(D_X,\opposite{D_X})$-bimodule is defined to be a module over that monad.
  $\mathcal{O}_X$ is naturally a $D_X$ module. So $\Gamma_{\Delta}(\mathcal{O} \boxtimes \mathcal{O})$
  has a natural structure of a $(D_X,\opposite{D_X})$-bimodule.
\end{remark}

\begin{remark} \label{explicitq}
  We can also define $Q$ by assembling the functors
  \[Q^{(m,n)} := \tilde{p}_{1,*}\tilde{p}_2^* : \Gamma_{\Delta}(\QC{X^{m+1}}) \to \Gamma_{\Delta}(\QC{X^{n+1}})\]
  into the functor
  \[Q : \colim_{\bm{\opposite{\Delta}}_s}(\Gamma_{\Delta}(\QC{X^{m+1}}),*) \to \lim_{\bm{\Delta}_s}(\Gamma_{\Delta}(\QC{X^{n+1}}),*)\]
  where $p_1,p_2$ are the two projection maps of
  \[X^{m+n+2} \cong X^{n+1} \times X^{m+1}\]
  so that we have the functors
  \[\tilde{p}_{1,*} : \Gamma_{\Delta}(\QC{X^{m+n+2}}) \to \Gamma_{\Delta}(\QC{X^{n+1}})\]
  \[\tilde{p}_2^* : \Gamma_{\Delta}(\QC{X^{m+1}}) \to \Gamma_{\Delta}(\QC{X^{m+n+2}})\]
  One can see gives the same functor as above for instance by computing the associated bimodule.
\end{remark}
Now we would like to construct an inverse to $Q$. Consider the following functor
\[R: \lim_{\bm{\Delta}_s}(\Gamma_{\Delta}(\QC{X^{n+1}}),*) \to \lim_{\bm{\Delta}_s}(\Gamma_{\Delta}(\QC{X^{n+1}}),\times)\]
which we define by assembling
\[R^{(n)}(\mathcal{F}) := \mathcal{F} \otimes_{\mathcal{O}_{X^{n+1}}} \omega_X^{\boxtimes n+1}\]
This obviously commute with the transition maps by Theorem 2.9 of \cite{previous}
because we have thrown away the degeneracy maps (by restricting to $\bm{\Delta_s}$).
We note that $R$ is a colimit-preserving functor with associated bimodule
\[\Gamma_{\Delta}(\omega_X \boxtimes \omega_X)\]
By inspection of the associated bimodules, we have
\begin{proposition}
  $R$ and $Q$ are self-dual under $\Mod{\QC{S}}^L$ duality (left-right duality).
\end{proposition}

The left-right switch is the following theorem.
\begin{theorem} \label{leftrightswitch}
  $R$ is the inverse functor of $Q$, and therefore
  \[\Mod{D_X} \cong \Mod{\opposite{D_X}}\]
\end{theorem}
\begin{proof}
  We show that $RQ \cong \id$, then the result will follow by duality.
  We will show this by directly computing $RQ$ using Remark \ref{explicitq}.

  Consider $RQ$ as a functor
  \[RQ:\colim_{\bm{\opposite{\Delta}}_s}(\Gamma_{\Delta}(\QC{X^{m+1}}),*) \to \lim_{\bm{\Delta}_s}(\Gamma_{\Delta}(\QC{X^{n+1}}),\times)\]
  then we can regard it as assembled from functors
  \[(RQ)^{(m,n)} :\Gamma_{\Delta}(\QC{X^{m+1}}) \to \Gamma_{\Delta}(\QC{X^{n+1}})\]
  which are defined by
  \[(RQ)^{(m,n)}(\mathcal{F}) \cong \omega^{\boxtimes n+1} \otimes_{\mathcal{O}_{X^{n+1}}} \tilde{p}_{1,*}\tilde{p}_2^*\mathcal{F}\]
  where
  \[\tilde{p}_{1,*} : \Gamma_{\Delta}(\QC{X^{m+n+2}}) \to \Gamma_{\Delta}(\QC{X^{n+1}})\]
  \[\tilde{p}_2^* : \Gamma_{\Delta}(\QC{X^{m+1}}) \to \Gamma_{\Delta}(\QC{X^{m+n+2}})\]
  then the claim follows from
  \[\omega^{\boxtimes n+1} \otimes_{\mathcal{O}_{X^{n+1}}} \tilde{p}_{1,*}\tilde{p}_2^*\mathcal{F} \cong \tilde{p}_{1,*}\tilde{p}_2^{\times}\mathcal{F}\]
  together with Theorem \ref{explicit3}.
\end{proof}

\section{Pushforward and Pullback of $D$-Modules} \label{pushpullsection}
We discuss in this section how to pushforward and pullback $D$-modules, both left and right.
We take the perspective of defining the functors on the category of $D$-modules first, and then
subsequently defining the transfer bimodules using those functors. Therefore, transfer modules take a back-seat in our story,
and we approach these functors as for crystals.

Suppose $f : X \to Y$ is a map of spectral schemes over $S$, both finite tor-amplitude, locally almost of finite presentation and separated over $S$.
Let us define pullback of $D_X$ modules, using the presentation of $\Mod{D_X}$ as a cosimplicial limit.
We define the functors
\[f^{+,(n)} : \Gamma_{\Delta}(\QC{Y^{n+1}}) \to \Gamma_{\Delta}(\QC{X^{n+1}})\]
by
\[f^{+,(n)}:= \Gamma_{\Delta}(f^*)^{n+1}i_{\Delta} = \widetilde{(f^*)^{n+1}}\]
where the second equality is just a notation.
These functors are obviously compatible with the transition maps, so they assemble into the functor
\[f^+ : \lim_{\bm{\Delta}_s}(\Gamma_{\Delta}(\QC{Y^{n+1}}),*) \to \lim_{\bm{\Delta}_s}(\Gamma_{\Delta}(\QC{X^{n+1}}),*)\]
or equivalently
\[f^+ : \Mod{D_Y} \to \Mod{D_X}\]
Which is what we call pullback of $D_X$ modules.

Left-right duality takes the pullback functor of $D_X$ modules to
the pushforward of $\opposite{D_X}$ modules, which we can also easily define directly. Consider the functors
\[f_+^{(n)} : \Gamma_{\Delta}(\QC{X^{n+1}}) \to \Gamma_{\Delta}(\QC{Y^{n+1}})\]
defined by 
\[f_+^{(n)}:= \Gamma_{\Delta}(f_*)^{n+1}i_{\Delta}=\widetilde{(f_*)^{n+1}}\]
where as before the second equality is a notation.
These assembles into the functor
\[f_+ : \colim_{\bm{\opposite{\Delta}}_s}(\Gamma_{\Delta}(\QC{X^{n+1}}),*) \to \colim_{\bm{\opposite{\Delta}}_s}(\Gamma_{\Delta}(\QC{Y^{n+1}}),*)\]
or equivalently
\[f_+ : \Mod{\opposite{D_X}} \to \Mod{\opposite{D_Y}}\]

Now we will define the transfer module to compare with the classical story.
As $f_+$ and $f^+$ are guaranteed to be colimit-preserving, these
functors have corresponding transfer modules. The transfer module for $f^+$ and the one
for $f_+$ will be the same (up to swapping the order of $X$ and $Y$)
because they are related by left-right duality. As in the section on the left-right switch,
we find the transfer module by considering the composition
\[G_{D_X}f^+F_{D_Y}\]
which simplifies to (by Remark \ref{explicit2})
\[f^*\tilde{\pi}^{(Y \times Y)}_{1,\times}\tilde{\pi}_2^{(Y \times Y),*}\]
Define
\[\delta_f : X \to X \times_S Y\]
to be the graph of $f$. Let us denote by $\Gamma_{f}$
the local cohomology functor on $\QC{X \times_S Y}$ relative to this subset.
Consider the split-exact sequence of presentable stable categories
\[\Gamma_{\Delta}(\QC{Y \times Y}) \to \QC{Y \times Y} \to \QC{U}\]
for the closed subset $\Delta$ in $Y \times Y$, where $U$ is the complement of $\Delta$.
We can apply the functor $\QC{X} \otimes_{\QC{Y}} \_$ to the above (where $\QC{Y}$ acts on the left) to 
get the split-exact sequence (see Remark B.7 of \cite{previous})
\[\QC{X} \otimes_{\QC{Y}} \Gamma_{\Delta}(\QC{Y \times Y}) \to \QC{X \times Y} \to \QC{V}\]
where $V$ is the complement of the graph of $f$ in $X \times Y$.
Therefore, we have the result
\begin{lemma}
  \[\Gamma_f(\QC{X \times_S Y}) \cong \QC{X} \otimes_{\QC{Y}} \Gamma_{\Delta}(\QC{Y \times_S Y})\]
  where $\QC{Y}$ acts on $\Gamma_{\Delta}(\QC{Y \times_S Y})$ via $\tilde{\pi}_1^*$.
\end{lemma}
With the description of $\Gamma_f(\QC{X \times_S Y})$ above, we have (by comparing their right adjoints)
\footnote{Note here both $\tilde{\pi}_1^{\times}$'s are colimit-preserving because their left adjoints are compact object preserving.}
\[\tilde{\pi}^{X \times Y}_{1,\times} \cong \id_{\QC{X}} \otimes \tilde{\pi}^{(Y \times Y)}_{1,\times}\]
Consider the diagram
\[
  \begin{tikzcd}
    X \times_S Y  \arrow{r}{f \times \id} \arrow{dd}{\pi^{(X \times Y)}_1} & Y \times_S Y \arrow{dd}{\pi_1^{(Y \times Y)}} \\ \\
    X \arrow{r}{f} & Y\\
  \end{tikzcd}
\]
Using Theorem A.1 of \cite{previous}, we have

\begin{equation*}
  \begin{split}
    f^*\tilde{\pi}^{(Y \times Y)}_{1,\times}\tilde{\pi}_2^{(Y \times Y),*}
    &\cong \tilde{\pi}^{(X \times Y)}_{1,\times}\widetilde{(f \times \id)}^*\tilde{\pi}^{(Y \times Y),*}_2\\
    &\cong \tilde{\pi}^{(X \times Y)}_{1,\times}\tilde{\pi}_2^{(X \times Y),*} \\
  \end{split}
\end{equation*}
where
\[\tilde{\pi}^{(X \times Y)}_{1,\times} : \Gamma_f(\QC{X \times_S Y}) \to \QC{X}\]
is defined as before (as the left adjoint of $\tilde{\pi}^{(X \times Y),*}_1$).
Hence the bimodule for the pullback functor $f^+$ is the one corresponding to the functor
\[\tilde{\pi}^{(X \times Y)}_{1,\times}\tilde{\pi}_2^{(X \times Y),*}\]
which is (by the left-right duals of Theorem 2.8 and Theorem 2.5 of \cite{previous})
\[\Gamma_f(\mathcal{O}_X \boxtimes \omega_Y)\]
Hence,
\begin{theodef} \label{transfermod}
  The transfer module $D_{X \to Y}$ for $f^+$ (and also $f_+$) is 
  \[D_{X \to Y/S} := \Gamma_f(\mathcal{O}_X \boxtimes \omega_Y) \cong \tilde{\pi}_1^{X \times Y, \times}\mathcal{O}_X \in \Gamma_f(\QC{X \times_S Y})\]
\end{theodef}
\begin{corollary}
  \[D_{X \to Y/S} \cong \widetilde{(f \times \id)}^*D_{Y/S}\]
  where
  \[\widetilde{(f \times \id)}^* : \Gamma_{\Delta}(\QC{Y \times_S Y}) \to \Gamma_f(\QC{X \times_S Y})\]
  is induced from the pullback functor $(f \times \id)^*$.
\end{corollary}

It is clear that $D_{X \to Y/S}$ naturally carries a left $D_{X/S}$ action and a right $D_{Y/S}$ action.
As the plus pullback functors compose well, also the transfer modules must compose well.
\begin{theorem}
  \[D_{X \to Z} \cong D_{X \to Y} \star_{D_Y} D_{Y \to Z}\]
\end{theorem}
\begin{remark}
  The star product is used here to recall
  that the algebra structure on $D$ is with respect to the convolution tensor product; but this can just be thought of as a tensor over $D_Y$
\end{remark}

\begin{remark} \label{adjointremark}
  Suppose additionally that $f$ is finite tor-amplitude. In this situation, we can define
  the functor
  \[f_{\dagger} : \Mod{D_X} \to \Mod{D_Y}\]
  as a functor
  \[f_{\dagger} : \colim_{\bm{\opposite{\Delta}_s}}(\QC{X^{m+1}},\times) \to \colim_{\bm{\opposite{\Delta}_s}}(\QC{Y^{m+1}},\times)\]
  which is given simply by assembling
  \[f_{\dagger}^{(n)} \cong \widetilde{f^{n+1}_{!}} : \QC{X^{m+1}} \to \QC{Y^{m+1}}\]
  and dually we can define the functor
  \[f^{\dagger} : \Mod{\opposite{D_Y}} \to \Mod{\opposite{D_X}}\]
  as a functor
  \[f^{\dagger} : \lim_{\bm{\Delta}_s}(\QC{Y^{m+1}},\times) \to \lim_{\bm{\Delta}_s}(\QC{X^{m+1}},\times)\]
  given simply given by assembling
  \[f^{\dagger,(n)} \cong \widetilde{f^{n+1,!}} : \QC{Y^{m+1}} \to \QC{X^{m+1}}\]
  If $f$ is in addition proper, then $f_{\dagger}$ is left adjoint to $f^+$ and
  $f^{\dagger}$ is right adjoint to $f_+$.
  The $!$-functors are defined in \cite{previous}.
  We note that these constructions are made easier because we restricted to $\bm{\Delta}_s$.
\end{remark}

\begin{remark}
  The dagger functors above are exactly the left-right switches (not dual!) of the plus functors.
  We must warn the reader here that the pullback and pushforward functors described above differ
  from the standard presentations, even when the notation is the same!
  For example, if the reader is comparing to the \cite{HTT2} book, the translation goes as follows for a
  map $f: X \to Y$ between smooth varieties
  \[\int_{f} \cong f_+[\dim{X} - \dim{Y}]\]
  and 
  \[f^{\dagger}_{HTT} \cong f^+[\dim{X} - \dim{Y}]\]
  where the left hand side is the in notation of \cite{HTT2}.
\end{remark}

\section{Kashiwara's Equivalence}
In this section, we prove a version of Kashiwara's equivalence in our context. In particular it shows that the category
of $D_Z$-modules on a singular variety $Z$ embedded in a smooth variety $X$ is equivalent to the subcategory of $D_X$-modules
supported on $Z$.

Suppose $z : Z \to X$ is a finite tor-amplitude closed subscheme, where $p_X : X \to S$ and $p_Z :Z \to S$ satisfy the our standing assumptions.
For any $n$, consider the split-exact sequence of presentable stable categories in $\Mod{\QC{S}}^L$
\begin{equation} \label{eeb}
  \Gamma_{\Delta_Z}(\QC{X^{n+1}}) \xrightarrow{i^{(n)}_Z} \Gamma_{\Delta_X}(\QC{X^{n+1}}) \xrightarrow{j^{(n),*}} \Gamma_{\Delta_U}(\QC{U^{n+1}})
\end{equation}

whose right adjoints $\Gamma^{(n)}_Z$ and $j^{(n)}_*$ (for $(i_Z^{(n)})_*$ and $j^{(n),*}$ respectively) are colimit-preserving.

In fact, using lower star (quasicoherent pushforward) as the transition maps, each of the above categories fits into a simplicial diagram.
Each of the functors $i^{(n)},j^{(n),*},\Gamma_Z^{(n)},j^{(n)}_*$ commute with the transition maps, therefore
after taking colimits of the simplicial diagrams of categories with $n$-th objects as described above, we recover a split-exact
sequence by Lemma B.6 of \cite{previous}.
\begin{equation} \label{eea}
\colim_{\opposite{\bm{\Delta_s}}}\Gamma_{\Delta_Z}(\QC{X^{n+1}})
  \xrightarrow{i_Z} \colim_{\opposite{\bm{\Delta_s}}}\Gamma_{\Delta_X}(\QC{X^{n+1}})
  \xrightarrow{j^*} \colim_{\opposite{\bm{\Delta_s}}}\Gamma_{\Delta_U}(\QC{U^{n+1}})
\end{equation}

Note that to check the last condition in the lemma, it suffices to check the exactness for compact objects in $\colim_{\opposite{\bm{\Delta_s}}}\Gamma_{\Delta}(\QC{X^{n+1}})$
where it follows from the split-exactness of (\ref{eeb}).
Let us define
\[\Gamma_Z(\Mod{\opposite{D_X}}):= \colim_{\bm{\opposite{\Delta}}}\Gamma_{\Delta_Z}(\QC{X^{n+1}})\]
The above sequence (\ref{eea}) can then be written simply as
\[\Gamma_Z(\Mod{\opposite{D_X}})
  \xrightarrow{i_Z} \Mod{\opposite{D_X}} 
  \xrightarrow{j^*} \Mod{\opposite{D_U}}\]
Therefore, $\Gamma_Z(\Mod{\opposite{D_X}})$ is the category of $\opposite{D}_X$ modules supported on $Z$ (as it is exactly those
that vanish after pulling back to $U$--note that pulling back right $D_X$ modules
to open subsets is well-defined). 

By modifying the proof of Theorem \ref{bigtheorem} we can see that the natural map
\[\Gamma_Z(\QC{X}) \to \Gamma_Z(\Mod{\opposite{D_X}})\]
is the left adjoint in a monadic adjunction where the monad in question is 
\[\tilde{\pi}_{1,*}\tilde{\pi}_2^{\times}\]
The functors 
\[\tilde{\pi}_2^{\times}: \Gamma_Z(\QC{X}) \to \Gamma_{\Delta_Z}(\QC{X \times X})\]
and
\[\tilde{\pi}_{1,*}: \Gamma_{\Delta_Z}(\QC{X \times X}) \to \Gamma_Z(\QC{X}) \]
are defined analogously to before, but restricted only to quasicoherent sheaves supported over $Z$.

We would like to compare the category $\Gamma_Z(\Mod{\opposite{D_X}})$
with
\[\Mod{\opposite{D_Z}} \cong \colim_{\bm{\opposite{\Delta}}}(\Gamma_{\Delta_Z}(\QC{Z^{n+1}}),*) \]
The functor
\[z_+ : \Mod{\opposite{D_Z}} \to \Mod{\opposite{D_X}}\]
naturally induces (and in fact factors through) the functor
\[\tilde{z}_+: \Mod{\opposite{D_Z}} \to \Gamma_Z(\Mod{\opposite{D_X}})\]
This is the functor which we will shows is an equivalence.
By Remark \ref{adjointremark}, $\tilde{z}_+$ has a colimit-preserving right adjoint, $\tilde{z}^{\dagger}$.
We wish to show that $\tilde{z}^{\dagger}\tilde{z}_+ \cong \id$ and $\tilde{z}^{\dagger}\tilde{z}_+\cong \id$.
First, notice that
\begin{lemma}
  The adjunction $\tilde{z}_+ \dashv \tilde{z}^{\dagger}$ is monadic.
\end{lemma}
\begin{proof}
  By Barr-Beck-Lurie it suffices to show that $\tilde{z}^{\dagger}$ is conservative and colimit-preserving. The functor
  is colimit-preserving directly from the definition given in Remark \ref{adjointremark}. For conservativeness, notice that
  \[G_{\opposite{D}_Z}\tilde{z}^{\dagger} \cong \tilde{z}^{\times}G_{\Gamma_Z(\opposite{D}_X)}\]
  and both of the functors on the right are conservative (see Lemma C.2 of \cite{previous}).
\end{proof}
Therefore we have shown that $\Gamma_Z(\Mod{\opposite{D}_X})$ is monadic over $\Mod{\opposite{D}_Z}$ which is itself monadic over $\QC{Z}$.
Hence the equivalence can be shown by showing that the map of monads
\[G_{\opposite{D_Z}}F_{\opposite{D_Z}} \to
G_{\opposite{D_Z}}\tilde{z}^{\dagger}\tilde{z}_+F_{\opposite{D_Z}}\] is an isomorphism.
Consider the endofunctor
\[G_{\opposite{D_Z}}\tilde{z}^{\dagger}\tilde{z}_+F_{\opposite{D_Z}}\]
Because of the commutative diagram
\[
  \begin{tikzcd}
    \QC{Z} \arrow{r}{\tilde{z}_*} \arrow{dd}{F_{\opposite{D_Z}}} & \Gamma_Z(\QC{X}) \arrow{dd}{F_{\Gamma_Z(\opposite{D_X})}}  \\ \\
    \Mod{\opposite{D_Z}} \arrow{r}{\tilde{z}_+} & \Gamma_Z(\Mod{\opposite{D_X}}) \\
  \end{tikzcd}
\]
The above is also
\[\tilde{z}^{\times}G_{\Gamma_Z(\opposite{D_X})}F_{\Gamma_Z(\opposite{D_X})}\tilde{z}_*\]
or
\[\tilde{z}^{\times}\tilde{\pi}_{1,*}\tilde{\pi}_2^{\times}\tilde{z}_*\]

By base-changing the split-exact sequence 
\begin{equation}
  \Gamma_{\Delta}(\QC{X \times X}) \to \QC{X} \to \QC{U}
\end{equation}
where $U$ is the complement of the diagonal, we can show
\[\Gamma_{\Delta_Z}(\QC{X \times Z}) \cong \Gamma_{\Delta_X}(\QC{X \times X}) \otimes_{\QC{X}} \QC{Z}\]
Now because the action of $\QC{X}$ on $\QC{Z}$ factors through $\Gamma_Z(\QC{X})$, we have
\[\Gamma_{\Delta_Z}(\QC{X \times Z}) \cong \Gamma_{\Delta_Z}(\QC{X \times X}) \otimes_{\Gamma_Z(\QC{X})} \QC{Z} \]
since
\[\Gamma_{\Delta_X}(\QC{X \times X}) \otimes_{\QC{X}} \Gamma_Z(\QC{X}) \cong \Gamma_{\Delta_Z}(\QC{X \times X})\]
Looking at the diagram
\[
  \begin{tikzcd}
    \QC{Z} \arrow{r}{\tilde{z}_*} \arrow{dd}{\tilde{\pi}_2^{(X \times Z),\times}} & \Gamma_Z(\QC{X}) \arrow{dd}{\tilde{\pi}_2^{\times}}  \\ \\
    \Gamma_{\Delta_Z}\QC{X \times Z} \arrow{r}{\widetilde{\id \times z}_*} & \Gamma_{\Delta_Z}(\QC{X \times X}) \\
  \end{tikzcd}
\]
we see that, since the $\tilde{\pi}_2^{\times}$ on the left is the base-change of the $\tilde{\pi}_2^{\times}$ on the right (by comparing their
left adjoints), this diagram commutes by Theorem A.1 of \cite{previous}. Hence
\[\tilde{\pi}_2^{\times}\tilde{z}_* \cong \widetilde{(\id \times z)}_*\tilde{\pi}_2^{(X \times Z),\times}\]
On the other side, we have the analogous commutative diagram
\[
  \begin{tikzcd}
    \Gamma_{\Delta_Z}\QC{X \times X} \arrow{r}{\widetilde{(z \times \id)}^{\times}} \arrow{dd}{\tilde{\pi}_{1,*}}
    & \Gamma_{\Delta_Z}(\QC{Z \times X}) \arrow{dd}{\tilde{\pi}^{(Z \times X)}_{1,*}}  \\ \\
    \Gamma_Z(\QC{X}) \arrow{r}{\tilde{z}^{\times}} & \QC{Z} \\
  \end{tikzcd}
\]
By Theorem A.1 of \cite{previous}, we have (this result is basically right adjoint to above)
\[\tilde{z}^{\times}\tilde{\pi}_{1,*}\cong \tilde{\pi}^{(Z \times X)}_{1,*}\widetilde{(z \times \id)}^{\times}\]
Hence
\[\tilde{z}^{\times}\tilde{\pi}_{1,*}\tilde{\pi}_2^{\times}\tilde{z}_*
\cong \tilde{\pi}^{(Z \times X)}_{1,*}\widetilde{(z \times \id)}^{\times}\widetilde{(\id \times z)}_*\tilde{\pi}_2^{(X \times Z),\times}\]
One can check that the natural map above
\[G_{\opposite{D_Z}}F_{\opposite{D_Z}} \to
  G_{\opposite{D_Z}}\tilde{z}^{\dagger}\tilde{z}_+F_{\opposite{D_Z}}\] is the same as the natural map
\[\tilde{\pi}^{(Z \times X)}_{1,*}\widetilde{(\id \times z)}_*\widetilde{(z \times \id)}^{\times}\tilde{\pi}_2^{(X \times Z),\times} \to 
\tilde{\pi}^{(Z \times X)}_{1,*}\widetilde{(z \times \id)}^{\times}\widetilde{(\id \times z)}_*\tilde{\pi}_2^{(X \times Z),\times}\]
coming from adjunction. Now consider the diagram
\[
  \begin{tikzcd}
    \Gamma_{\Delta_Z}\QC{X \times Z} \arrow{r}{\widetilde{(z \times \id)}^{\times}} \arrow{dd}{\widetilde{(\id \times z)}_*}
    & \Gamma_{\Delta_Z}(\QC{Z \times Z}) \arrow{dd}{\widetilde{(\id \times z)}_*}  \\ \\
    \Gamma_Z(\QC{X \times X}) \arrow{r}{\widetilde{(z \times \id)}^{\times}} & \Gamma_Z(\QC{Z \times X}) \\
  \end{tikzcd}
\]
Using the fact that $Z$ is a closed subscheme, we can show that set-theoretically, 
\[(Z \times X) \cap (X \times Z) = (Z \times Z)\]
inside $X \times X$. Therefore, we have the categorical isomorphism
\[\Gamma_{\Delta_Z}\QC{Z \times Z} \cong \Gamma_{\Delta_Z}\QC{X \times Z} \otimes_{\Gamma_{\Delta_Z}\QC{X \times X}}\Gamma_{\Delta_Z}\QC{Z \times X}\]
(where the action is via pullback not convolution). Now the isomorphism
\[\widetilde{(z \times \id)}^{\times}\widetilde{(\id \times z)}_* \cong \widetilde{(\id \times z)}_*\widetilde{(z \times \id)}^{\times}\]
follows from the the diagram above via Theorem A.1 of \cite{previous}.
Hence we have shown
\begin{theorem}[Kashiwara's Equivalence] \label{kashi}
  Suppose $z: Z \to X$ is a finite tor-amplitude closed subscheme, and $p_X$, $p_Z$ satisfy
  our standing assumptions, then
  \[\tilde{z}_+: \Mod{\opposite{D_Z}} \to \Gamma_Z(\Mod{\opposite{D_X}})\]
  is an equivalence of categories with inverse $\tilde{z}^{\dagger}$.
\end{theorem}
By left-right duality, we also have
\begin{corollary} \label{kashi2}
  Suppose $z: Z \to X$ is a finite tor-amplitude closed subscheme, and $p_X$,$p_Z$ satisfy
  our standing assumptions, then
  \[\tilde{z}^+: \Gamma_Z(\Mod{D_X}) \to \Mod{D_Z}\]
  is an equivalence of categories with inverse $\tilde{z}_{\dagger}$.
\end{corollary}

\begin{remark}
  In the case that $X = S = \Spec{R}$ for $R$ a discrete ring,
  the ring of differential operators on $Z=\Spec{R/I}$ is simply
  $D_Z = Hom_R(R/I,R/I)$. Theorem \ref{kashi} then follows from derived Morita theory
  since $R/I$ is a compact generator of
  the category of $I$-nilpotent $R$ modules (In the sense of Theorem 7.1.1.6 of \cite{SAG}).
\end{remark}  

\section{Base Change and Proper Finite Tor-amplitude Descent}
In this section, we prove descent of the category of $D$-modules with respect to 
proper, finite tor-amplitude, surjective morphisms. As a consequence, we deduce fppf
descent for $D$-modules for nonderived schemes. In this section we will denote by $X^{(n)}_T$ the $n$-fold
product of $X$ over $T$ for a scheme $X$ over $T$.

\begin{proposition} \label{descend1}
  Suppose $S$ is a homologically bounded spectral Noetherian scheme and $X$ and $Y$
  are locally almost of finite presentation, finite tor-amplitude, and separated over $S$. Also suppose
  that there is a map $f: X \to Y$ which is finite tor-amplitude. Then,
  there is a natural isomorphism
  \[\Gamma_X(\QC{X \times_S Y}) \cong \lim_{\bm{\Delta_s}}(\Gamma_X(\QC{X \times_S (X^{(n+1)}_Y)}))\]
  with the transition maps either upper star or upper cross.
\end{proposition}
\begin{proof}
  We will assume the transition maps are upper star, the proof for upper cross is entirely analogous.
  We will apply \cite{HA} Corollary 4.7.5.3. Condition (1) is automatic (for upper cross it follows from
  the fact that the map $f: X \to Y$ is finite tor-amplitude). For (2), we need to check that, for any map $[m] \to [n]$,
  the square below is left adjointable. Here the horizontal arrow is upper star along the map induced by projection map
  to all but the first component for the product over $Y$ and the vertical maps are upper star along the map
  induced by the map $[m] \to [n]$. 
  \[
    \begin{tikzcd}
      \Gamma_X(\QC{X \times_S (X^{(m+1)}_Y)}) \arrow{r}{\widetilde{\id \times \pi_{\hat{1}}}^*} \arrow{d}{u} & \Gamma_X(\QC{X \times_S (X^{(m+2)}_Y)}) \arrow{d}{v} \\
      \Gamma_X(\QC{X \times_S (X^{(n+1)}_Y)}) \arrow{r}{\widetilde{\id \times \pi_{\hat{1}}}^*} & \Gamma_X(\QC{X \times_S (X^{(n+2)}_Y)})
    \end{tikzcd}
  \]
  We need to check the map
  \[(\widetilde{\id \times \pi_{\hat{1}}})_{\times}v \to (\widetilde{\id \times \pi_{\hat{1}}})_{\times}v(\widetilde{\id \times \pi_{\hat{1}}})^*(\widetilde{\id \times \pi_{\hat{1}}})_{\times}
    \cong (\widetilde{\id \times \pi_{\hat{1}}})_{\times}(\widetilde{\id \times \pi_{\hat{1}}})^*u(\widetilde{\id \times \pi_{\hat{1}}})_{\times} \to u(\widetilde{\id \times \pi_{\hat{1}}})_{\times}\]
  is an isomorphism.
  We have the isomorphisms (which can again be verified using split-exact sequences)
  \begin{equation*}
    \begin{split}
      \Gamma_X(\QC{X \times_S (X^{(n+2)}_Y)})
      &\cong \Gamma_X(\QC{X \times_S X \times_Y X^{(n+1)}_Y}) \\
      &\cong \Gamma_X(\Gamma_X(\QC{X \times_S X^{(n+1)}_Y}) \otimes_{\QC{X \times_S X^{(m+1)}_Y}} \QC{X \times_S X^{(m+2)}_Y}) \\
      &\cong \Gamma_X(\QC{X \times_S X^{(m+1)}_Y}) \otimes_{\Gamma_X(\QC{X \times_S X^{(n+1)}_Y})} \Gamma_X(\QC{X \times_S X^{(n+2)}_Y}) \\
    \end{split}
  \end{equation*}
  using these isomorphisms we have
  \[v \cong u \otimes \id\]
  and we can therefore easily see the map above is an isomorphism (because we can separate out the vertical and horizontal maps
  and hence the isomorphism is coming from the adjunction data).

  Lastly, we need to check that the map
  \[\Gamma_X(\QC{X \times_S Y}) \to \Gamma_X(\QC{X \times_S X})\]
  is conservative, but that follows from Lemma C.2 of \cite{previous}.
\end{proof}

\begin{corollary}
  Suppose $S$ is a homologically bounded spectral Noetherian scheme and $X$ and $Y$
  are locally almost of finite presentation, finite tor-amplitude, and separated over $S$. Also suppose
  that there is a map $f: X \to Y$ which is finite tor-amplitude. Then,
  there is a natural isomorphism
  \[\Gamma_{\Delta}(\QC{X \times_S X}) \otimes_{\Gamma_{\Delta}(\QC{X \times_Y X})} \QC{X} \cong \Gamma_{X}(\QC{X \times_S Y})\]
\end{corollary}
\begin{proof}
  Using Proposition \ref{resolve}, we have the isomorphism
  \[\Gamma_{\Delta}(\QC{X \times_S X}) \otimes_{\Gamma_{\Delta}(\QC{X \times_Y X})} \QC{X} \]
  \[\cong \colim_{\bm{\Delta}_s}(\Gamma_{\Delta}(\QC{X \times_S X}) \otimes_{\Gamma_{\Delta}(\QC{X \times_Y X})} \Gamma_X(\QC{X^{(n+2)}_Y}))\]
  By applying Lemma \ref{tensorlemma}, we also have the isomorphism
  \[\Gamma_{\Delta}(\QC{X \times_S X}) \otimes_{\Gamma_{\Delta}(\QC{X \times_Y X}} \Gamma_X(\QC{X^{(n+2)}_Y}) \]
  \[\cong \Gamma_X(\QC{X \times_S X}) \otimes_{\QC{X}} \Gamma_X(\QC{X^{(n+1)}_Y}) \cong \Gamma_X(\QC{X \times_S X^{(n+1)}_Y})\]
  Hence the result follows by switching the colimit to a limit (by taking right adjoints of the transition maps) and applying proposition \ref{descend1}. 
\end{proof}

\begin{proposition}
  Suppose $S$ is a homologically bounded spectral Noetherian scheme and $X$ and $Y$
  are locally almost of finite presentation, finite tor-amplitude, and separated over $S$. Also suppose
  that there is a map $f: X \to Y$ which is finite tor-amplitude. Then,
  there is a natural isomorphism
  \[(\widetilde{\id \times f})_*\left(\Gamma_X(\mathcal{O}_X \boxtimes \omega_{X/S}) \otimes_{D_{X/Y}} \mathcal{O}_X\right) \cong \Gamma_X(\mathcal{O}_X \boxtimes \omega_{Y/S})\]
  of objects in $\Gamma_{X}(\QC{X \times_S Y})$ respecting the $(D_X,D_Y)$-module structure. Note that the left hand side is a bit of an abuse of notation because the object
  \[\Gamma_X(\mathcal{O}_X \boxtimes \omega_{X/S}) \otimes_{D_{X/Y}} \mathcal{O}_X\]
  has no $\mathcal{O}_X$ action on the right, so the pushforward is really only as Zariski sheaves.
\end{proposition}
\begin{proof}
  We can make sense of an element 
  \[\Gamma_X(\mathcal{O}_X \boxtimes \omega_{X/S}) \otimes_{D_{X/Y}} \mathcal{O}_X \in \Gamma_X(\QC{X \times X}) \otimes_{\Gamma_X(\QC{X \times_Y X})} \QC{X}\]
  as follows. The right hand side is the colimit of the simplicial diagram
  \[\ldots \Gamma_X(\QC{X \times X}) \otimes \Gamma_X(\QC{X \times_Y X}) \otimes \QC{X} \righttwoarrow \Gamma_X(\QC{X \times X}) \otimes \QC{X}\]
  Therefore, the element
  \[\Gamma_X(\mathcal{O}_X \boxtimes \omega_{X/S}) \otimes \mathcal{O}_X \in \Gamma_X(\QC{X \times X}) \otimes \QC{X}\]
  maps to an element in $\Gamma_X(\QC{X \times X}) \otimes_{\Gamma_X(\QC{X \times_Y X})} \QC{X})$ and so does the element
  \[\Gamma_X(\mathcal{O}_X \boxtimes \omega_{X/S}) \otimes D_{X/Y} \otimes \mathcal{O}_X \in \Gamma_X(\QC{X \times X}) \otimes \Gamma_X(\QC{X \times_Y X}) \otimes \QC{X}\]
  and so on.
  As $D_{X/Y}$ acts on $\Gamma_X(\mathcal{O}_X \boxtimes \omega_{X/S})$ on the right and $\mathcal{O}_X$ on the right, there is a simplicial diagram
  \[\ldots \Gamma_X(\mathcal{O}_X \boxtimes \omega_{X/S}) \otimes D_{X/Y} \otimes \mathcal{O}_X \righttwoarrow \Gamma_X(\mathcal{O}_X \boxtimes \omega_{X/S}) \otimes \mathcal{O}_X\]
  whose colimit we call 
  \[\Gamma_X(\mathcal{O}_X \boxtimes \omega_{X/S}) \otimes_{D_{X/Y}} \mathcal{O}_X \in \Gamma_X(\QC{X \times X}) \otimes_{\Gamma_X(\QC{X \times_Y X})} \QC{X}\]
  We can check directly that the image in $\Gamma_X(\QC{X \times Y})$ of this element is
  \[(\widetilde{\id \times f})_*\left(\Gamma_X(\mathcal{O}_X \boxtimes \omega_{X/S}) \otimes_{D_{X/Y}} \mathcal{O}_X\right)\]
  and hence we just need to show this is isomorphic to
  \[\Gamma_X(\mathcal{O}_X \boxtimes \omega_{Y/S})\]
  Under the isomorphism
  \[\Gamma_X(\QC{X \times_S Y}) \cong \lim_{\bm{\Delta_s}}(\Gamma_X(\QC{X \times_S (X^{(n+1)}_Y)}),\times)\]
  the object
  \[\Gamma_X(\mathcal{O}_X \boxtimes \omega_{Y/S})\]
  is given under the isomorphism by the compatible system
  \[\left(\Gamma_X(\mathcal{O}_X \boxtimes \omega_{X^{(n+1)}_Y/S})\right)_n \in \lim_{\bm{\Delta_s}}(\Gamma_X(\QC{X \times_S (X^{(n+1)}_Y)}), \times )\]
  Now, using the isomorphism
  \[\Gamma_{\Delta}(\QC{X \times_S X}) \otimes_{\Gamma_{\Delta}(\QC{X \times_Y X})} \QC{X} \cong \lim_{\bm{\Delta_s}}(\Gamma_X(\QC{X \times_S (X^{(n+1)}_Y)}),\times)\]
  which is induced by the compatible system of functors
  \[\id \otimes \tilde{\pi}_1^{\times} : \Gamma_{\Delta}(\QC{X \times_S X}) \otimes_{\Gamma_{\Delta}(\QC{X \times_Y X})} \QC{X} \to \Gamma_X(\QC{X \times_S (X^{(n+1)}_Y)})\]
  the object
  \[\Gamma_X(\mathcal{O}_X \boxtimes \omega_{X/S}) \otimes_{D_{X/Y}} \mathcal{O}_X \in \Gamma_X(\QC{X \times X}) \otimes_{\Gamma_X(\QC{X \times_Y X})} \QC{X}\]
  corresponds to (following a direct calculation)
  \[\left(\Gamma_X(\mathcal{O}_X \boxtimes \omega_{X^{(n+1)}_Y/S})\right)_n \in \lim_{\bm{\Delta_s}}(\Gamma_X(\QC{X \times_S (X^{(n+1)}_Y)}), \times )\]
  and so the proof is complete.
\end{proof}
\begin{corollary}
  Suppose $T$ and $Y$ are spectral schemes over a homologically bounded base spectral Noetherian scheme $S$, such that the structure maps
  are finite tor-amplitude, locally almost of finite presentation and separated. 
  Let $f: T \to Y$ be proper and finite tor-amplitude. Then, the diagram below commutes
  \[
    \begin{tikzcd}
      \Mod{\opposite{D_{T/S}}} \arrow{r}{g_{+}} \arrow{d}{\Phi_{T \to Y}} & \Mod{\opposite{D_{Y/S}}} \arrow{d}{G_{\opposite{D_Y}}} \\
      \Mod{\opposite{D_{T/Y}}} \arrow{r}{g_{+}} & \Mod{\opposite{D_{Y/Y}}}
    \end{tikzcd}
  \]
  in the sense the natural map 
  \[g_+\Phi_{T \to Y} \to g_+ \Phi_{T \to Y} g^{\dagger} g_+ \cong g_+ g^{\dagger} G_{\opposite{D_Y}} g_{+} \to G_{\opposite{D_Y}} g_+\]
  is an isomorphism (this is an example of a Beck-Chevalley condition)
  where
  \[\Phi_{T \to Y} : \Mod{D_{T/S}} \to \Mod{D_{T/Y}}\]
  can be defined by assemblying the quasicoherent upper cross maps
  \[\Gamma_{T}(\QC{T^{(n+1)}_S}) \to \Gamma_{T}(\QC{T^{(n+1)}_Y})\]
  using the isomorphisms
  \[\Mod{\opposite{D_{T/S}}} \cong \lim_{\bm{\Delta}_s}(\Gamma_{\Delta}(\QC{T^{(n+1)}_S}),\times)\]
  and
  \[\Mod{\opposite{D_{T/Y}}} \cong \lim_{\bm{\Delta}_s}(\Gamma_{\Delta}(\QC{T^{(n+1)}_Y}),\times)\]
  Note we have abused notation to use $g_{+}$ to denote two maps.
\end{corollary}
\begin{proof}
  Without loss of generality assume all the schemes are affine. Now let us express the natural map as a morphism of transfer modules, namely,
  \[\mathcal{O}_T \otimes_{\opposite{D_{T/Y}}} \opposite{D_{T/S}} \to
    \mathcal{O}_T \otimes_{\opposite{D_{T/Y}}} \opposite{D_{T/S}} \otimes_{\opposite{D_{T/S}}} \Gamma_T(\omega_{T/S} \boxtimes \mathcal{O}_Y)
    \otimes_{\opposite{D_{Y/S}}} \Gamma_T(\omega_{Y/S} \boxtimes \mathcal{O}_T)\]
  \[\cong \mathcal{O}_T \otimes_{\opposite{D_{T/Y}}} \omega_{T/Y} \otimes_{\mathcal{O}_Y} \opposite{D_{Y/S}} \otimes_{\opposite{D_{Y/S}}} \Gamma_T(\omega_{Y/S} \boxtimes \mathcal{O}_T) \to
    \opposite{D_{Y/S}} \otimes_{\opposite{D_{Y/S}}} \Gamma_T(\omega_{Y/S} \boxtimes \mathcal{O}_T)\]
  which boils down to the above proposition.

\end{proof}
  
\begin{proposition}
  Suppose $S$ is a homologically bounded spectral Noetherian scheme and $X$ and $Y$
  are locally almost of finite presentation, finite tor-amplitude, and separated over $S$. Also suppose
  that there is a map $f: X \to Y$ which is finite tor-amplitude. Then,
  there is a natural isomorphism
  \[(\widetilde{\id \times f})_*\left(\Gamma_X(\omega_{X/S} \boxtimes \mathcal{O}_{X}) \otimes_{\opposite{D_{X/Y}}} \omega_{X/Y}\right) \cong \Gamma_X(\omega_{X/S} \boxtimes \mathcal{O}_Y)\]
  of objects in $\Gamma_{X}(\QC{X \times_S Y})$ respecting the right $D$-module structures. Note that the left hand side is a bit of an abuse of notation as before.
\end{proposition}
\begin{proof}
  The element
  \[(\widetilde{\id \times f})_*\left(\Gamma_X(\omega_{X/S} \boxtimes \mathcal{O}_{X}) \otimes_{\opposite{D_{X/Y}}} \omega_{X/Y}\right)\]
  can be expressed by the tensor
  \[\Gamma_X(\omega_{X/S} \boxtimes \mathcal{O}_{X}) \otimes_{\opposite{D_{X/Y}}} \omega_{X/Y} \in \Gamma_X(\QC{X \times X}) \otimes_{\Gamma_X(\QC{X \times_Y X})} \QC{X}\]
  and under the isomorphism of Proposition \ref{descend1}
  is given by the compatible system
  \[\left(\Gamma_X(\omega_{X/S} \boxtimes \omega_{X^{(n+1)}_Y/S})\right)_n \in \lim_{\bm{\Delta_s}}(\Gamma_X(\QC{X \times_S (X^{(n+1)}_Y)}), \times )\]
  The right hand side is given by the element
  \[\left(\Gamma_X(\omega_{X/S} \boxtimes \mathcal{O}_{X^{(n+1)}_Y/S})\right)_n \in \lim_{\bm{\Delta_s}}(\Gamma_X(\QC{X \times_S (X^{(n+1)}_Y)}), *)\]
  which also corresponds to the element
  \[\left(\Gamma_X(\omega_{X/S} \boxtimes \omega_{X^{(n+1)}_Y/S})\right)_n \in \lim_{\bm{\Delta_s}}(\Gamma_X(\QC{X \times_S (X^{(n+1)}_Y)}), \times )\]
  under the isomorphism
  \[\lim_{\bm{\Delta_s}}(\Gamma_X(\QC{X \times_S (X^{(n+1)}_Y)}), *) \cong \lim_{\bm{\Delta_s}}(\Gamma_X(\QC{X \times_S (X^{(n+1)}_Y)}), \times )\]
\end{proof}

\begin{corollary} \label{pushinvar}
  Suppose $T$ and $Y$ are spectral schemes over a homologically bounded base spectral Noetherian scheme $S$, such that the structure maps
  are finite tor-amplitude, locally almost of finite presentation and separated. 
  Let $f: T \to Y$ be proper and finite tor-amplitude. Then, the diagram below commutes
  \[
    \begin{tikzcd}
      \Mod{D_{T/S}} \arrow{r}{g_{\dagger}} \arrow{d}{\Phi_{T \to Y}} & \Mod{D_{Y/S}} \arrow{d}{G_{D_Y}} \\
      \Mod{D_{T/Y}} \arrow{r}{g_{\dagger}} & \Mod{D_{Y/Y}}
    \end{tikzcd}
  \]
  in the sense the natural map 
  \[g_{\dagger}\Phi_{T \to Y} \to g_{\dagger} \Phi_{T \to Y} g^+ g_{\dagger} \cong g_{\dagger} g^+ G_{D_Y} g_{\dagger} \to G_{D_Y} g_{\dagger}\]
  is an isomorphism (this is an example of a Beck-Chevalley condition)
  where
  \[\Phi_{T \to Y} : \Mod{D_{T/S}} \to \Mod{D_{T/Y}}\]
  can be defined by assemblying the quasicoherent pullback maps
  \[\Gamma_{T}(\QC{T^{(n+1)}_S}) \to \Gamma_{T}(\QC{T^{(n+1)}_Y})\]
  using the isomorphisms
  \[\Mod{D_{T/S}} \cong \lim_{\bm{\Delta}_s}(\Gamma_{\Delta}(\QC{T^{(n+1)}_S}),*)\]
  and
  \[\Mod{D_{T/Y}} \cong \lim_{\bm{\Delta}_s}(\Gamma_{\Delta}(\QC{T^{(n+1)}_Y}),*)\]
  Note we have abused notation to use $g_{\dagger}$ to denote two maps.
\end{corollary}
\begin{proof}
  Without loss of generality assume all the schemes are affine. Now let us express the natural map as a morphism of transfer modules, namely,
  \[\omega_{T/Y} \otimes_{D_{T/Y}} D_{T/S} \to
    \omega_{T/Y} \otimes_{D_{T/Y}} D_{T/S} \otimes_{D_{T/S}} \Gamma_T(\mathcal{O}_T \boxtimes \omega_{Y/S}) \otimes_{D_{Y/S}} \Gamma_T(\mathcal{O}_Y \boxtimes \omega_{T/S})\]
  \[\cong \omega_{T/Y} \otimes_{D_{T/Y}} \mathcal{O}_T \otimes_{D_{Y/Y}} D_{Y/S} \otimes_{D_{Y/S}} \Gamma_T(\mathcal{O}_Y \boxtimes \omega_{T/S}) \to
    D_{Y/S} \otimes_{D_{Y/S}} \Gamma_T(\mathcal{O}_Y \boxtimes \omega_{T/S})\]
  which boils down to the above proposition.
\end{proof}

\begin{remark}
  The core of the two corollaries above is the statement that $D$-module pushforward is in some sense invariant of the choice of the base $S$.
\end{remark}
\begin{proposition} \label{pushpullforDmod}
  Suppose $X$, $Y$, and $T$ are spectral schemes over a base homologicaly bounded spectral Noetherian scheme $S$, such that the structure maps
  are finite tor-amplitude, locally almost of finite presentation and separated. 
  Let $g: T \to Y$ be proper and finite tor-amplitude and $f: X \to Y$ be
  a morphism. Then, the diagram below commutes
  \[
    \begin{tikzcd}
      \Mod{D_{T/S}} \arrow{r}{g_{\dagger}} \arrow{d}{(f')^+} & \Mod{D_{Y/S}} \arrow{d}{f^+} \\
      \Mod{D_{T \times_Y X/S}} \arrow{r}{(g')_{\dagger}} & \Mod{D_{X/S}}
    \end{tikzcd}
  \]
  in the sense that the natural map
  \[(g')_{\dagger}(f')^+ \to (g')_{\dagger}(f')^+g^+g_{\dagger} \cong (g')_{\dagger}(g')^+f^+g_{\dagger} \to f^+g_{\dagger}\]
  is an isomorphism (this is an example of a Beck-Chevalley condition).
\end{proposition}
\begin{proof}
  We can check the natural map is an isomorphism after post-composition with
  \[G_{D_X} : \Mod{D_{X/S}} \to \QC{X}\]
  since $G_{D_X}$ is conservative.
  By Corollary \ref{pushinvar} the following diagram commutes
  \[
    \begin{tikzcd}
      \Mod{D_{T \times_Y X/S}} \arrow{r}{(g')_{\dagger}} \arrow{d}{\Phi_{T \times_Y X}} & \Mod{D_{X/S}} \arrow{d}{G_{D_X}} \\
      \Mod{D_{T \times_Y X/X}} \arrow{r}{(g')_{\dagger}} & \Mod{D_{X/X}}
    \end{tikzcd}
  \]
  in the sense the natural map
  \[(g')_{\dagger}\Phi_{T \times_Y X} \to (g')_{\dagger} \Phi_{T \times_Y X} (g')^+ (g')_{\dagger} \cong
    (g')_{\dagger}(g')^+G_{D_X}(g')_{\dagger} \to G_{D_X}(g')_{\dagger}\]
  is an isomorphism where
  \[\Phi_{T \times_Y X} : D_{T \times_Y X/S} \to D_{T \times_Y X/X}\]
  is assembled from the quasicoherent pullback maps
  \[\Gamma_{T \times_Y X}(\QC{(T \times_Y X)^{(n+1)}_S}) \to \Gamma_{T \times_Y X}(\QC{(T \times_Y X)^{(n+1)}_X})\]
  using the isomorphisms
  \[\Mod{D_{T \times_Y X/S}} \cong \lim_{\bm{\Delta}_s}(\QC{\Gamma_{T \times_Y X}(\QC{(T \times_Y X)^{(n+1)}_S})},*)\]
  and
  \[\Mod{D_{T \times_Y X/X}} \cong \lim_{\bm{\Delta}_s}(\QC{\Gamma_{T \times_Y X}(\QC{(T \times_Y X)^{(n+1)}_X})},*)\]
  Hence it suffices to show that the following diagram commutes (with the Beck-Chevalley map)
  \[
    \begin{tikzcd}
      \Mod{D_{T/S}} \arrow{r}{(g')_{\dagger}} \arrow{d}{\Phi_{T \times_Y X}(f')^+} & \Mod{D_{Y/S}} \arrow{d}{G_{D_X}f^+} \\
      \Mod{D_{T \times_Y X/X}} \arrow{r}{(g')_{\dagger}} & \Mod{D_{X/X}}
    \end{tikzcd}
  \]
  Since the diagram
  \[
    \begin{tikzcd}
      \Mod{D_{T/S}} \arrow{r}{(g')_{\dagger}} \arrow{d}{\Phi_{T}} & \Mod{D_{Y/S}} \arrow{d}{G_{D_Y}} \\
      \Mod{D_{T/Y}} \arrow{r}{(g')_{\dagger}} & \Mod{D_{Y/Y}}
    \end{tikzcd}
  \]
  commutes (with the Beck-Chevalley maps) by Corollary \ref{pushinvar}, it suffices to show that the diagram
  \[
    \begin{tikzcd}
      \Mod{D_{T/Y}} \arrow{r}{(g')_{\dagger}} \arrow{d}{(f')^*} & \Mod{D_{Y/Y}} \arrow{d}{f^*} \\
      \Mod{D_{T \times_Y X/X}} \arrow{r}{(g')_{\dagger}} & \Mod{D_{X/X}}
    \end{tikzcd}
  \]
  commutes with the Beck-Chevalley maps. This holds since the entire theory base-changes well
  with respect to the base. More precisely, there are isomorphisms
  \[\Mod{D_{T \times_Y X/X}} \cong \Mod{D_{T/Y}} \otimes_{\QC{Y}} \QC{X}\]
  \[\Mod{D_{X/X}} \cong \Mod{D_{Y/Y}} \otimes_{\QC{Y}} \QC{X}\]
  and the functor $(g')_{\dagger}$ on the bottom is the base-change of the functor $(g')_{\dagger}$
  on top.

\end{proof}
  
\begin{theorem}\label{pftdescent}
  The category of $D$-modules satisfies descent along proper finite tor-amplitude surjective maps.
\end{theorem}
\begin{proof}
  Suppose $f : T \to X $ is locally almost of finite presentation, proper, finite tor-amplitude, and surjective over some base $S$ so that the situation satisfies the standing assumptions. We wish to show the map
  \[\Mod{D_X} \to \lim_{\bm{\Delta}}(\Mod{D_{T^{(n)}_X}})\]
  is an isomorphism, where the transition maps are $D$-module pullback ($^+$-pullback) and $T^{(n)}_X$ is the
  $n$-fold (derived) cartesian product of $T$ over $X$. We will apply Corollary 4.7.5.3 of \cite{HA}. We need to check three conditions
  \begin{enumerate}
  \item $\Mod{D_X}$ admits geometric realizations of $f^+$-split simplicial objects and those geometric realizations are
    preserved by $f^+$.
  \item For every morphism $[m] \to [n]$ in $\bm{\Delta}_+$, the diagram
    \[
      \begin{tikzcd}
        \Mod{D_{T^{(m)}_X}} \arrow{r}{d^0} \arrow{d} & \Mod{D_{T^{(m+1)}_X}} \arrow{d} \\
        \Mod{D_{T^{(n)}_X}} \arrow{r}{d^0} & \Mod{D_{T^{(n+1)}_X}}
      \end{tikzcd}
    \]
    is left-adjointable (see \cite{HA} 4.7.4.13)
    where $d^0 : [N] \to [N+1]$ denotes the map which sends $k$ to $k+1$ for $k \in [N]$.

  \item $f^+$ is conservative.
  \end{enumerate}

  (1) is automatic. (2) is a direct application of Proposition \ref{pushpullforDmod}. (3)
  follows from Lemma C.2 of \cite{previous} (reducing to the discrete setting) and \cite{AGrass} Theorem 11.12 (\textit{h}-descent with derived
  Čech nerve).
\end{proof}
\begin{corollary}
  Suppose $S$ is an underived Noetherian scheme, then the category of $D$-modules on finite-type,
  finite tor-amplitude, separated $S$-schemes (relative to $S$) satisfies fppf-descent.
\end{corollary}
\begin{proof}
  Follows from étale descent (Theorem \ref{bigtheorem2}) and finite-flat descent (Theorem \ref{pftdescent}) together with \cite{stacks} Lemma 0DET.
\end{proof}

\section{Comparison with the De Rham Stack for Underived Noetherian Schemes} \label{cryssec}
In this section, we discuss the relationship between $D$-modules as defined in the previous sections and
the more classical story of quasicoherent sheaves on the de Rham stack. The latter is the same thing
as quasi-coherent crystals on the (big) infinitesimal site. Over characteristic zero, all the results below appear in \cite{CryD}.

Suppose $S$ is a underived Noetherian scheme.
Let us denote by $SCH^{ft}_{/S}$ the category of all finite-type separated underived schemes over $S$.
For any finite-type morphism $X \to Y$ in $SCH^{ft}_{/S}$, we can define 
\begin{definition} \label{derhamstack}
  The relative de Rham stack $(X/Y)_{dR}$ is the presheaf on $SCH^{ft}_{/S}$ defined by
  \[(X/Y)_{dR}(U) := \Hom(U_{red},X) \times_{\Hom(U_{red},Y)} \Hom(U,Y)\]
  where $U_{red}$ is the reduced subscheme $U$ and the Hom's are computed in $SCH^{ft}_{/S}$.
\end{definition}
In other words, it is the presheaf of maps from $U$ to $Y$ such that on $U_{red}$ the map lifts to $X$.
This presheaf is in fact a sheaf on the Zariski (or étale) topology.
A reminder to the reader that we use the terms presheaf/sheaf to mean presheaf/sheaf of spaces, in the sense of \cite{HTT}.

The (contravariant) functor taking a scheme to its
category of quasicoherent sheaves
\[QCoh : \opposite{SCH^{ft}_{/S}} \to \widehat{Cat_{\infty}}\]
is a sheaf of categories on $SCH^{ft}_{/S}$, with respect to
either the Zariski or étale topology. 
Hence, we can define $QCoh$ for any presheaf on $SCH^{ft}_{/S}$ by
\[QCoh(\mathcal{F}) := \Hom(\mathcal{F},QCoh)\]
where the Hom is taken in the category of presheaves
of categories on $SCH_{/S}$. Alternatively, we can think of this as defining $QCoh$
via Kan extension. We note that this agrees with the definition given in \cite{stacks} Tag 0H0H,
as the difference in the choice of sites does not matter here.

\begin{proposition} \label{closedstackequi}
  Suppose $X \to Y$ is a closed immersion and $Y$ is finite-type over $S$, then
  \[QCoh((X/Y)_{dR}) \cong \Gamma_X(\QC{Y})\]
\end{proposition}
\begin{proof}
  Let $\mathcal{I}$ be the ideal sheaf corresponding to the subscheme $X$ in $Y$. Then we can define
  \[X^{(n)}:= \underline{\Spec_Y{\mathcal{O}_Y/\mathcal{I}^n}}\]
  the $n$-th level thickening of $X$ in $Y$. Then we can see that as sheaves on $SCH^{ft}_{/S}$ with the Zariski or étale topology,
  \[(X/Y)_{dR} \cong \colim_n X^{(n)}\]
  Hence,
  \[QCoh((X/Y)_{dR}) \cong \lim_n \QC{X^{(n)}}\]
  By Proposition 0H0L in \cite{stacks}, we know this category is affine-locally naturally
  isomorphic to the category of derived $\mathcal{I}$-complete modules. By the Greenlees-May
  isomorphism, it is also locally naturally isomorphic to the category of derived $\mathcal{I}$-torsion modules.
  Hence, it is true globally as well by Zariski descent.
\end{proof}

\begin{theorem} \label{stackequi} 
  Suppose $X \to Y$ is a finite tor-amplitude and separated map in $SCH^{ft}_{/S}$, then
  there is a natural isomorphism
  \[\Mod{D_{X/Y}} \cong QCoh((X/Y)_{dR})\]
\end{theorem}
\begin{proof}
  Let us construct a natural map in the forward direction which we will then show is an isomorphism.
  We can write the right hand side explicitly as
  \[\lim_{\{(Z,f)| Z \in SCH^{ft}_{/S}, f: Z \to (X/Y)_{dR}\}}{\QC{Z}}\]
  Now, by Theorem 11.12 in \cite{AGrass}, this is also
  \[\lim_{\{(Z,f)| Z \in SCH^{ft}_{/S}, f: Z \to (X/Y)_{dR}\}}{\big(\lim(\QC{Z_{red}} \righttwoarrow \QC{Z_{red} \times_Z Z_{red}} \rightthreearrow \ldots)\big)}\]
  where here the fibre products are derived. Given a map
  \[f : Z \to (X/Y)_{dR}\]
  we can extract from it maps
  \[f_{red}: Z_{red} \to X\]
  \[f_{base}: Z \to Y\]
  and hence create maps
  \[g^{(n)}: Z_{red} \times_Z \ldots \times_Z Z_{red} \to X \times_Y \ldots \times_Y X\]
  Therefore as
  \[\Mod{D_{X/Y}} \cong \lim(\QC{X} \righttwoarrow \Gamma_{\Delta}(\QC{X \times_Y X}) \rightthreearrow \ldots)\]
  We obtain a map (depending on the pair $(Z,f)$)
  \[\Mod{D_{X/Y}} \to {\lim(\QC{Z_{red}} \righttwoarrow \QC{Z_{red} \times_Z Z_{red}} \rightthreearrow \ldots)}\]
  which are compatible with varying $(Z,f)$. This gives the desired map in the forward direction.

  It remains to show the map constructed above is an isomorphism. By Zariski descent we can reduce to when $X$, $Y$, and $S$ are all affine. In this case, we can view $X$ as
  a closed subscheme of $Y' \cong \mathbb{A}^n_Y$. Now, the map
  \[(X/Y')_{dR} \to (X/Y)_{dR}\]
  is surjective as discrete Zariski sheaves as $Y'$ is smooth over $Y$. Hence, it is an effective epimorphism
  of Zariski sheaves on $SCH^{ft}_{/S}$ by Proposition 7.2.1.14 of \cite{HTT} (effective epimorphisms can be detected on $\pi_0$)
  and the fact that epimorphisms of discrete sheaves are effective (Thm IV.7.8 in \cite{sheafgeo}).
  Therefore, there's a natural isomorphism
  \[QCoh((X/Y)_{dR}) \cong \lim((QCoh(X/Y')_{dR}) \righttwoarrow QCoh((X/(Y' \times_Y Y'))_{dR}) \rightthreearrow \ldots)\]
  We can check that the closed immersion $X \to Y'$ is finite tor-amplitude by Proposition 6.1.2.3 of \cite{SAG},
  reducing to the case where $S$ and $Y$ are both fields. In this case, $Y'$ is Spec of a polynomial algebra over a field. Hence
  as $X$ is homologically bounded (finite tor-amplitude over $Y$), the map $X \to Y'$ is finite tor-amplitude. Therefore,
  we also have the isomorphism
  \[\Mod{D_{X/Y}} \cong \lim(\Gamma_X(\QC{Y'}) \righttwoarrow \Gamma_X(\QC{Y' \times_Y Y'}) \rightthreearrow \ldots)\]
  coming from Kashiwara's equivalence (Corollary \ref{kashi2}) and the limit presentation of the category of $D$-modules
  in Theorem \ref{bigtheorem2}. Now, we can see that the map constructed above
  \[\Mod{D_{X/Y}} \to QCoh((X/Y)_{dR})\]
  is induced from the following diagram by taking limits horizontally.
  \[
    \begin{tikzcd}
      \Gamma_X(\QC{Y'}) \arrow[r, yshift=0.7ex] \arrow[r, yshift=-0.7ex] \arrow{d}
      & \Gamma_X(\QC{Y' \times_Y Y'}) \arrow[r, yshift=1.2ex] \arrow[r] \arrow[r, yshift=-1.2ex] \arrow{d} & \ldots \\
      QCoh((X/Y')_{dR}) \arrow[r, yshift=0.7ex] \arrow[r, yshift=-0.7ex]
      & QCoh((X/(Y' \times_Y Y'))_{dR}) \arrow[r, yshift=1.2ex] \arrow[r] \arrow[r, yshift=-1.2ex] & \ldots \\
    \end{tikzcd}
  \]
  All the vertical arrows are isomorphisms by Proposition \ref{closedstackequi} and the claim follows.
\end{proof}
\begin{remark}
  It is an interesting question to ask if analogous results hold if we enlarge the site
  $SCH^{ft}_{/S}$ to include spectral schemes and allow all the schemes to spectral.
  We do not presently know the answer to this question.
\end{remark}

\section{Universal Homeomorphisms and Relation with \cite{Cusp}} \label{cuspsec}
In this section, we discuss an analogous result to Kashiwara's Equivalence for universal homeomorphisms
and describe an application of our work to recover some results of \cite{Cusp}.
Let $S$ be a underived Noetherian scheme and $X$ be a finite-type underived $S$-scheme
such that $p_X : X \to S$ is finite tor-amplitude and separated. We additionally suppose that there is a
universal homeomorphism $\tau: \tilde{X} \to X$ of separated underived Noetherian schemes, such that the composition
$p_{\tilde{X}} := p_X\tau$ is finite tor-amplitude.

Let us denote by ${\tilde{X}}^{(m+1)_X}$ the $(m+1)$-fold (derived) product of $\tilde{X}$ over $X$. We'll need the following lemma.
(Note that we can still define $\Gamma_{\Delta}$ on homologically unbounded schemes, so that the following statement is well-defined)
\begin{lemma} \label{extendedAG}
  For any $n \ge 0$,
  \begin{equation*}
    \lim_{[m] \in \bm{\Delta}}(\Gamma_{\Delta}(\QC{({\tilde{X}}^{(m+1)_X})^{(n+1)}}),*) \cong \Gamma_{\Delta}(\QC{X^{(n+1)}})
  \end{equation*}
\end{lemma}
\begin{proof}
  We start with Theorem 11.12 in \cite{AGrass}, which implies
  \[\lim_{[m] \in \bm{\Delta}}(\QC{({\tilde{X}}^{(m+1)_X})^{(n+1)}},*) \cong \QC{X^{(n+1)}}\]
  via the cover $\tilde{X}^{(n+1)} \to X^{(n+1)}$. Now, the right hand side of the lemma is the full subcategory of
  \[\QC{X^{(n+1)}}\]which vanishes away from the diagonal.
  Then, the result follows from the fact that
  \[({\tilde{X}}^{(m+1)_X})^{(n+1)} \to  X^{(n+1)}\]
  is a homeomorphism and
  \[\Gamma_{\Delta}(\QC{({\tilde{X}}^{(m+1)_X})^{(n+1)}})\]
  is the full subcategory of 
  \[\QC{({\tilde{X}}^{(m+1)_X})^{(n+1)}})\]
  which vanishes away from the diagonal.
\end{proof}

\begin{theorem} \label{homeokashi}
  The functor
  \[\tau^+ : \Mod{D_X} \to \Mod{D_{\tilde{X}}}\]
  is an equivalence of categories.
\end{theorem}
\begin{proof}
  We first exhibit a functor in the reverse direction. We start from the isomorphism
  \begin{equation}\label{startingpoint}
    \Mod{D_{\tilde{X}}} \cong \lim_{\bm{\Delta}}(\Gamma_{\Delta}(\QC{{\tilde{X}}^{n+1}}),*)
  \end{equation}
  Now, we have a functor
  \[\times : \bm{\Delta} \times \bm{\Delta} \to \bm{\Delta}\]
  such that
  \[\times ([m],[n]) = [m] \times [n] \cong [mn+m+n]\]
  where we order $[m] \times [n]$ by lexicographic ordering. Hence, we have a functor
  \begin{equation}\label{secondpoint}
    \lim_{\bm{\Delta}}(\Gamma_{\Delta}(\QC{{\tilde{X}}^{n+1}}),*) \to \lim_{\bm{\Delta} \times \bm{\Delta}}(\Gamma_{\Delta}(\QC{{\tilde{X}}^{(m+1)(n+1)}}),*)
  \end{equation}
  Because of the natural map
  \[{\tilde{X}}^{(m+1)_X} \to {\tilde{X}}^{(m+1)}\]
  where the product on the right hand side is over $S$, there is a natural pull-back map
  \[\QC{{\tilde{X}}^{(m+1)(n+1)}} \to \QC{({\tilde{X}}^{(m+1)_X})^{(n+1)}} \]
  This induces a functor
  \begin{equation}\label{thirdpoint}
    \lim_{\bm{\Delta} \times \bm{\Delta}}(\Gamma_{\Delta}(\QC{{\tilde{X}}^{(m+1)(n+1)}}),*) \to \lim_{\bm{\Delta} \times \bm{\Delta}}(\Gamma_{\Delta}(\QC{({\tilde{X}}^{(m+1)_X})^{(n+1)}}),*)
  \end{equation}
  Finally, by Lemma \ref{extendedAG}, we have 
  \begin{equation}\label{fourthpoint}
    \lim_{\bm{\Delta} \times \bm{\Delta}}(\QC{({\tilde{X}}^{(m+1)_X})^{(n+1)}},*) \cong \lim_{\bm{\Delta}}(\QC{X^{(n+1)}},*)
  \end{equation}
  Combining (\ref{startingpoint}), (\ref{secondpoint}), (\ref{thirdpoint}), and (\ref{fourthpoint}), we can construct a functor
  \begin{equation}
    \tau_{-} : \Mod{D_{\tilde{X}}} \to \Mod{D_X}
  \end{equation}
  Now we can check that $\tau_{-}\tau^+ \cong \id$ and $\tau^+\tau_{-} \cong \id$ by computing
  transfer modules of the composites (by computing the image of $D_X$ and $D_{\tilde{X}}$ respectively).
\end{proof}
\begin{remark}
  The transfer module of $\tau_{-}$ is $\Gamma_{\Delta}(\mathcal{O}_X \boxtimes \omega_{\tilde{X}})$ and
  the transfer module of $\tau^+$ is $\Gamma_{\Delta}(\mathcal{O}_{\tilde{X}} \boxtimes \omega_X)$.
\end{remark}
\begin{remark}
  For underived Noetherian schemes, it is also possible to prove a stronger version of Kashiwara's equivalence for non
  finite tor-amplitude closed immersions using a similar approach.
\end{remark}

By left-right duality, we also have
\begin{corollary}
  \[z_+ : \Mod{\opposite{D_{\tilde{X}}}} \to \Mod{\opposite{D_X}}\]
  is an equivalence of categories.
\end{corollary}

To compare our results with those of \cite{Cusp}, let us recall their setup. Assuming for the rest of this section that
$S = \Spec{k}$ where $k$ is a field, $X$ and $\tilde{X}$ are Cohen-Macaulay $k$-varieties of dimension $d$, and finally that
\[H^1(\Gamma_{\Delta}(M \boxtimes \omega_{\tilde{X}})) = 0\]
and
\[H^1(\Gamma_{\Delta}(M \boxtimes \omega_{X})) = 0\]
for all $M \in \QC{X}^{[0,0]}$, so that $\tau$ is a \textit{good} cuspidal quotient between
\textit{good} Cohen-Macaulay varieties in the terminology of \textit{loc. cit.}
\begin{lemma}
  In the above situation,
  \[H^i(\Gamma_{\Delta}(M \boxtimes \omega_{\tilde{X}})) = 0\]
  and
  \[H^i(\Gamma_{\Delta}(M \boxtimes \omega_{X})) = 0\]
  for all $i \neq 0 $ and $M \in \QC{X}^{[0,0]}$
\end{lemma}
\begin{proof}
  Without loss of generality, we can assume that $X$ and $\tilde{X}$ are affine. Namely, $X = \Spec{R}$ and $\tilde{X} = \Spec{\tilde{R}}$.
  Let $\pi_1 : X \times \tilde{X} \to X$ be the projection to the first component. Then, there is an isomorphism (by Theorem 2.6 in \cite{previous})
  \[\Gamma_{\Delta}(M \boxtimes \omega_{\tilde{X}}) \cong \tilde{\pi}_1^{\times}M\]
  We can rewrite this as
  \[\colim_n{\Hom_{R \otimes_k \tilde{R}}((R \otimes_k \tilde{R})/I^n,\Hom_R(R \otimes_k \tilde{R},M))} \cong
    \colim_n{\Hom_R((R \otimes_k \tilde{R})/I^n,M)}\]
  where $I$ is the kernel of the surjection $R \otimes_k \tilde{R} \to \tilde{R}$. Hence, we can see that
  for injective (discrete) $M$, $\Gamma_{\Delta}(M \boxtimes \omega_{\tilde{X}})$ is discrete. Using the assumptions,
  we can then conclude using injective resolutions that for all discrete $M$, $\Gamma_{\Delta}(M \boxtimes \omega_{\tilde{X}})$ is discrete. The
  second claim follows similarly.
\end{proof}

\begin{theorem}[Theorem 1.2 in \cite{Cusp}] \label{recover}
  In the above situation, there is a Morita equivalence between the (sheaf of) algebras
  $H^0(D_{\tilde{X}})$ and $H^0(D_X)$ induced by
  \[H^0(D_{\tilde{X} \rightarrow X}) \cong H^0(\Gamma_{\Delta}(\mathcal{O}_{\tilde{X}} \boxtimes \omega_X))\]
  and
  \[H^0(D_{\tilde{X} \leftarrow X}) := H^0(\Gamma_{\Delta}(\mathcal{O}_{X} \boxtimes \omega_{\tilde{X}}))\]
\end{theorem}
\begin{proof}
  Without the $H^0$'s, this is simply Theorem \ref{homeokashi}. Hence, it suffices to show all the $H^0$'s above are redundant, because
  the objects are already in degree $0$ under our assumptions. But this follows from the Grothendieck-Sato formula (Corollary 2.11 of \cite{previous})
  and the above lemma.
\end{proof}

\section{Relation with Hochschild Cohomology} \label{hochc}
In this section, we discuss a decategorification of Corollary \ref{keycoro} in the case
$X = \Spec{A}$ is a smooth affine variety over $S = \Spec{k}$, which we assume to be affine and discrete (concentrated in $\pi_0$). Namely,
we will show a result of the form
\[D_A \cong A \otimes_{H} A\]
for $H$ being the $E_2$ ring of Hochschild cohomology of $A$, where $D_A$ is the ring of differential operators on $\Spec{A}$.
Corollary \ref{keycoro} has been known since the work of Beraldo, in \cite{Ber1} and \cite{Ber2}, and we are heavily influenced by
those works. We will also allow $A$ to be noncommutative in this section, as it will not affect our proofs and may even
be helpful psychologically.

Suppose $A$ is an $E_1$ ring over $k$ (which no longer needs to be concentrated in $\pi_0$), which is compact in
the category of $A$-bimodules, $\Mod{(A \otimes \opposite{A})}$. This is a condition
that we have not assumed in the previous sections and is some sort of generalization of smoothness.
In fact, it is equivalent to $\Mod{A}$ being a smooth category, using Definition 4.5 in \cite{Perry}.
The Hochschild cohomology of $A$ over $k$ is the $E_2$ ring defined by
\begin{equation} \label{HCH}
  \HCH(A/k):=\Hom_{\End_k(\Mod{A} \, )}(\id,\id)
\end{equation}
where $\End_k(\Mod{A})$ is the monoidal category of $k$-linear endomorphisms of $\Mod{A}$.
Notice that
\[\End_k(\Mod{A}) \cong \Mod{(A \otimes \opposite{A})}\]
and therefore we also have 
\[\HCH(A/k) \cong \Hom_{A \otimes \opposite{A}}(A,A)\]
although it is harder to see the $E_2$ structure this way.
We establish a convention for the $E_2$ ring $\HCH(A/k)$. Using equation (\ref{HCH}), we call
the $E_1$ algebra structure on $\HCH(A/k)$ induced from the monoidal structure of $\End_k(\Mod{A})$ the horizontal product--$\mu_1$, and the
$E_1$ algebra structure induced from composition of morphisms in $\End_k(\Mod{A})$ the vertical product--$\mu_2$.
For
\[f,g \in \HCH(A/k)\]
$\mu_2(f,g)$ is the composition $fg$ in $\Hom_{\End_k(\Mod{A})}(\id,\id)$ and will be denoted by $f$ above $g$.
These two $E_1$ structures are compatible and also they are noncanonically isomorphic.
In particular we have the following coherence diagram
\[
  \begin{tikzcd}
    \substack{\HCH(A/k) \otimes \HCH(A/k) \\ \otimes \, \, \, \, \, \, \, \, \, \, \otimes \\ \HCH(A/k) \otimes \HCH(A/k)}  \arrow{rr}{\mu_2 \otimes \mu_2}
    \arrow{dd}{\substack{\mu_1 \\ \otimes \\ \mu_1}} && \HCH(A/k) \otimes \HCH(A/k) \arrow{dd}{\mu_1} \\ \\
    \substack{\HCH(A/k) \\ \otimes \\ \HCH(A/k)} \arrow{rr}{\mu_2} && \HCH(A/k)\\
  \end{tikzcd}
\]
Let us explain the notation. The vertical tensor product mean the same as horizontal tensor, but
the author finds it clearer to reserve writing the tensor product vertically when applying the vertical product.
The upper left term is just the tensor product of four copies of $\HCH(A/k)$, denoted as a square
for the reasons we just mentioned. Normally, for a $E_1$ ring, we can define left and right modules over it.
Because $\HCH(A/k)$ has vertical multiplication, we can also define up and down modules over it similarly.
We denote the category of modules of left modules over $\HCH(A/k)$
by
\[\Mod{\HCH(A/k)^{left}}\]
and similarly for right, up, and down modules. Each of these is a monoidal category where the monoidal structure is taken in an orthogonal direction.
In particular left modules (the module is to the right of the ring) have downwards monoidal products, etc.

Let us think of the multiplication in $A$, $\mu_A$ as being horizontal, so that we can form
left modules, right modules, and bimodules over $A$ naturally.
Then $\Mod{A \otimes \opposite{A}}$, the category of bimodules over $A$, is naturally
a monoidal category by tensoring over $A$ (we think of the monoidal product as happening in the horizontal direction.
Let $\Gamma_{\Delta}(\Mod{(A \otimes \opposite{A})})$ denote the subcategory of $\Mod{(A \otimes_k \opposite{A})}$
generated under colimits by $A$.
We can think of $\HCH(A/k)$ as a one object monoidal category where the endomorphisms
of that object is $\HCH(A/k)$ with $\mu_2$ product (and $\mu_1$ is responsible for the monoidal structure).
Then this monoidal category naturally maps into $\Gamma_{\Delta}(\Mod{(A \otimes \opposite{A})})$
as a map of monoidal categories, basically by definition, where the object maps to $A$. This induces a map of monoidal categories
\[\Phi: \Mod{\HCH(A/k)^{down}} \to \Gamma_{\Delta}(\Mod{A \otimes \opposite{A}})\]
which is an isomorphism because $A$ is a compact generator whose ring of endomorphisms is $\HCH(A/k)$ with $\mu_2$ product.

For a down $\HCH(A/k)$ module $M$,
$\Phi$ sends $M$ to the $A \otimes \opposite{A}$ module given by
\[\scalebox{1.3}{%
$\substack{M \\ \otimes\\ A} \scriptstyle{\HCH(A/k)}$}\]
where the $A$ on the bottom has commutating up $\HCH(A/k)$ action and
left and right $A$ actions, i.e. a left $A \otimes \opposite{A}$ action.
This gives an $A \otimes \opposite{A}$-module structure on the tensor product.
(The vertical tensor is a normal tensor product over the $E_1$ ring $\HCH(A/k)$ with
the $\mu_2$ product). We can think of the monoidal-ness of the functor $\Phi$ as follows. 
First note $A$ is a up $\HCH(A/k)$ algebra, because the evaluation map
\[A \otimes \Hom_{A \otimes \opposite{A}}(A,A) \to A\]
(which we think of as an up action because composition of functions in $\HCH{A/k}$ is visualized upwards)
is compatible with the horizontal monoidal product--tensoring over $A$. Therefore, we have coherence diagrams such as
\[
  \begin{tikzcd}
    \substack{\HCH(A/k) \otimes \HCH(A/k) \\ \otimes \, \, \, \, \, \, \, \, \, \, \otimes \\ A \otimes A}  \arrow{rr}
    \arrow{dd}{\substack{\mu_1 \\ \otimes \\ \mu_A}} && A \otimes A \arrow{dd}{\mu_A} \\ \\
    \substack{\HCH(A/k) \\ \otimes \\ A} \arrow{rr} && A\\
  \end{tikzcd}
\]
where the horizontal maps are the structure maps of $A$ as a up $\HCH(A/k)$ module.
Now, suppose $M$ and $N$ are both down $\HCH(A/k)$ modules. We can consider the tensor product
\[\Phi(M) \otimes_A \Phi(n) \cong 
\scalebox{1.3}{%
    $ \left(
      \substack{M \\ \otimes\\ A} \scriptstyle{\HCH(A/k)}\right)
    \otimes_A
  \left(
      \substack{N \\ \otimes\\ A} \scriptstyle{\HCH(A/k)}\right)
    $}\]
We rewrite this as
\[\scalebox{1.3}{%
    $ \left(
      \substack{M \\ \otimes\\ A} \scriptstyle{\HCH(A/k)}\right)
    \otimes_{
  \left(
      \substack{\HCH(A/k) \\ \otimes\\ A} \scriptstyle{\HCH(A/k)}\right)}
  \left(
      \substack{N \\ \otimes\\ A} \scriptstyle{\HCH(A/k)}\right)
      $}\]
We can instead evaluate this tensor horizontally first to get
\[\scalebox{1.3}{%
    $ \left(
      \substack{M \otimes_{\HCH} N\\ \otimes\\ A} \scriptstyle{\HCH(A/k)}\right)
    $}\]
which captures the fact that the functor $\Phi$ is monoidal. We note that
the horizontal actions of $\HCH(A/k)$ on $M$ and $N$ are coming from the monoidal structure
on $\Mod{\HCH(A/k)^{down}}$.

In the reverse direction, for an $A \otimes \opposite{A}$-module $N$, the
down $\HCH(A/k)$ module corresponding to $N$ is
\[\Psi(N) := \Hom_{A \otimes \opposite{A}}(A,N)\]
as this is the right adjoint of $\Phi$. $\Psi$ is also monoidal since $A$
is the unit of the monoidal structure on $\Mod{A \otimes \opposite{A}}$.
Pictorially, we can write an element of $\Psi(N)$ as a vertical map
\[
  \begin{tikzcd}
    N \\
    A \arrow{u} 
  \end{tikzcd}
\]
and the monoidal structure of $\Psi$ is seen by tensoring horizontally over $A$
\begin{equation} \label{monoidalness}
  \left(\begin{tikzcd}
    N \\
    A \arrow{u} 
  \end{tikzcd}\right) \otimes
  \left(\begin{tikzcd}
    N' \\
    A \arrow{u}
  \end{tikzcd}\right) \to
  \left(\begin{tikzcd}
    N \otimes_A N' \\
    A \arrow{u} 
  \end{tikzcd}\right)
\end{equation}
In fact there is also a left and right
$\HCH(A/k)$ naturally on $\Psi(N)$, because we can tensor (over $A$) an $A$-bimodule map from $A$ to $N$ on the left
or right with a $A$-bimodule map from $A$ to $A$. The left, down and right actions are compatible, in the sense that any of these actions can induce the
others by rotating the $E_1$ structure on $\HCH(A/k)$, assuming that we never cross the direction which makes the action into an up action. We can represent these actions
by the following cartoon.
\begin{equation} \label{cartoon}
  \begin{gathered}
  \begin{tikzcd}
    A \\
    A \arrow{u} 
  \end{tikzcd}
  \begin{tikzcd}
    \arrow[loop left]
  \end{tikzcd}
  \begin{tikzcd}
    N \\
    A \arrow{u} 
  \end{tikzcd}
  \begin{tikzcd}
    \arrow[loop right]
  \end{tikzcd}
  \begin{tikzcd}
    A \\
    A \arrow{u} 
  \end{tikzcd}
  \\
  \begin{tikzcd}
    \arrow[loop below] \\
  \end{tikzcd}
  \\
  \begin{tikzcd}
    A \\
    A \arrow{u}
  \end{tikzcd}
  \end{gathered}
\end{equation}
The fact that the drawn actions are compatible follows from the fact that
we can fill in more copies of $\HCH(A/k)$ in the lower left and lower right
corners, whose actions on their neighboring $\HCH(A/k)$'s is compatible with
the actions on $\Hom_k(A,N)$ indicated in the diagram. This makes is clear that
the map in (\ref{monoidalness}) is compatible with the actions of $\HCH(A/k)$.

Inside $\Gamma_{\Delta}(\Mod{A \otimes \opposite{A}})$,
there is the natural ring $D_A$ which we've encountered,
\[D_A := \Gamma_{\Delta}(\Hom_k(A,A))\]
which is here thought of as a ring with horizontal multiplication.
$D_A$ is sent to a down $\HCH{A/k}$ algebra by $\Psi$. To see which, we compute
\begin{equation*}
  \begin{split}
    \Hom_{A \otimes \opposite{A}}(A,D_A)
    &\cong \Hom_{A \otimes \opposite{A}}(A, \Gamma_{\Delta}(\Hom(A,A))) \\
    &\cong \Hom_{A \otimes \opposite{A}}(A, \Hom(A,A)) \\
    &\cong \Hom_{A \otimes \opposite{A}}(A, \Hom_{A}(A \otimes A,A)) \\
  \end{split}
\end{equation*}
where the action of $A$ on $A \otimes A$ in the last line is on the left multiplication on the left $A$.
The right $A \otimes \opposite{A}$ action on $A \otimes A$ is via acting on the left $A$ on the right and the right $A$ on the left
(which is a right $\opposite{A}$ action), inducing a left $A \otimes \opposite{A}$
module structure on $\Hom_{\opposite{A}}(A \otimes \opposite{A},A)$. Therefore,
\begin{equation*}
  \begin{split}
    \Hom_{A \otimes \opposite{A}}(A,D_A)
    &\cong \Hom_{A \otimes \opposite{A}}(A, \Hom_{A}(A \otimes A,A)) \\
    &\cong \Hom_{A}((A \otimes A) \otimes_{A \otimes \opposite{A}}A,A) \\
    &\cong \Hom_{A}(A,A) \\
    &\cong \opposite{A} \\
  \end{split}
\end{equation*}
where a direct check shows that the algebra structure on the last line is indeed the opposite
of the algebra structure on $A$.
We would like to figure out the down $\HCH(A/k)$ action. But before we do that, let's streamline the computation above to just
\begin{equation*}
  \begin{split}
    \Hom_{A \otimes \opposite{A}}(A,D_A)
    &\cong \Hom_{A \otimes \opposite{A}}(A, \Gamma_{\Delta}(\Hom(A,A))) \\
    &\cong \Hom_{A \otimes \opposite{A}}(A, \Hom(A,A)) \\
    &\cong \Hom_{A}(A \otimes_{A}A,A) \\
    &\cong \Hom_{A}(A,A) \\
    &\cong \opposite{A} \\
  \end{split}
\end{equation*}
where from line two to three, we think of the isomorphism as
the application of a single ``enriched'' tensor-hom adjunction with the $(k,A)$ bimodule
$A$, both actions are on the left, where the left $A$ actions on both
sides of the $\Hom$ just come along for the ride. Tensor-hom adjunction in this form is probably well-known, but
one can think of the computations above as justification for this ``enriched'' tensor-hom adjunction as well.
Let us draw a picture of this isomorphism.
\[
  \begin{gathered}
    \begin{tikzcd}
      \Hom_k(A \xleftarrow{} \textcolor{red}{A}) \\
      A \arrow[u,"A","A"']
    \end{tikzcd}
    \mapsto
    \begin{tikzcd}
      A \\
      A \otimes_A \textcolor{red}{A}\arrow[u,"A"]
    \end{tikzcd}
  \end{gathered}
\]
where the left arrow is labeled on both sides to indicate that it is required to be $(A,A)$-bilinear whereas the right
diagram only requires that the map is left $A$-linear. From this diagram it is clear that the left $\HCH(A/k)$ action
will be the most convenient to work with, because it is unfazed by the tensor-hom adjunction. Namely, it is simply the action
\begin{equation} \label{leftaction}
  \begin{gathered}
  \begin{tikzcd}
    A \\
    A \arrow[u,"A","A"'] 
  \end{tikzcd}
  \begin{tikzcd}
    \arrow[loop left]
  \end{tikzcd}
  \begin{tikzcd}
    A \\
    A \arrow[u,"A"]
  \end{tikzcd}
  \end{gathered}
\end{equation}
Our diagram therefore shows the left $\HCH{A/k}$ action, in fact it shows a left $\HCH{A/k}$ algebra structure.
We would like to drag it to a down $\HCH{A/k}$ algebra and describe it. First, let us start with the standard
action of $\HCH(A/k)$ on $A$, namely $A$ as an up $\HCH(A/k)$ algebra. We can visualize it like so
\[
  \begin{gathered}
    \begin{tikzcd}
      A \\
      A \arrow[u,"A","A"'] \\
      A \arrow[loop above]  \\
      A \otimes_k A \arrow[u,"A","A"']
  \end{tikzcd}
\end{gathered}
\]
By writing it this way, we see that indeed there are compatible actions, as in the diagram (\ref{cartoon})
\[
  \begin{gathered}
    \begin{tikzcd}
      A \\
      A \arrow[u,"A","A"']
    \end{tikzcd}
    \\
    \begin{tikzcd}
      A \\
      A \arrow[u,"A","A"']
    \end{tikzcd}
    \begin{tikzcd}
      \arrow[loop left]
    \end{tikzcd}
    \begin{tikzcd}
      A \arrow[loop above]  \\
      A \otimes_k A \arrow[u,"A","A"']
    \end{tikzcd}
    \begin{tikzcd}
      \arrow[loop right]
    \end{tikzcd}
    \begin{tikzcd}
      A \\
      A \arrow[u,"A","A"']
    \end{tikzcd}
\end{gathered}
\]
Therefore, we can deduce that the action in diagram (\ref{leftaction}) is the standard up $\HCH(A/k)$ algebra $A$
rotated by $90^{\circ}$ counterclockwise.
We visualize $\HCH(A/k)$ staying still and the module rotating around it.
To get to the down $\HCH{A/k}$ algebra, we further rotate by $90^{\circ}$ counterclockwise. Therefore, in total we have
\begin{equation}
  \Psi(D_A) \cong \Hom_{A \otimes \opposite{A}}(A,D_A) \cong A_{180^{\circ}}
\end{equation}
meaning that we drag the standard up $\HCH(A/k)$ algebra $180^{\circ}$ degrees counterclockwise to obtain a down $\HCH(A/k)$ algebra.
Note that doing this naturally reverses the order of multiplication on the ring, making the underlying ring $\opposite{A}$.
We note that the order of the dragging matters, and we do not even get the same underlying module if we drag in the opposite direction.
Using the inverse functor to $\Psi$, we have
\begin{equation} \label{decat}
  D_A \cong \scalebox{1.3}{%
    $\substack{A_{180^{\circ}} \\ \otimes\\ A} \scriptstyle{\HCH(A/k)}$}
\end{equation}

We can also define the opposite ring $\opposite{D_A}$, and by rotating equation (\ref{decat}) by $180^{\circ}$ clockwise, we can see that
\[\opposite{D_A} \cong \scalebox{1.3}{%
$\substack{A_{-180^{\circ}} \\ \otimes\\ A} \scriptstyle{\HCH(A/k)}$}\]
Since in general $D_A$ and $\opposite{D_A}$ are not canonically isomorphic even as $A$-bimodules,
we must conclude that dragging $A$ as a down $\HCH(A/k)$ module counterclockwise by one full rotation should genuinely yields a different
$\HCH(A/k)$ module in general.

Denote by $A_{90^{\circ}}$ the left $\HCH(A/k)$ algebra and $A_{-90^{\circ}}$ the right $\HCH(A/k)$ algebra
obtained by draggin the standard up $\HCH(A/k)$ algebra by the corresponding angles.
Then, by rotating the isomorphism (\ref{decat}) above by $90^{\circ}$ clockwise, we get
\[D_{A,-90^{\circ}} \cong A_{-90^{\circ}} \otimes_{\HCH(A/k)} A_{90^{\circ}}\]
We can categorify the above to get
\[\Mod{\opposite{D_A}} \cong \Mod{\opposite{A}} \otimes_{\Mod{\HCH(A/k)^{down}}} \Mod{A}\]
which was indeed what we intended to decategorify.

\appendix

\section{Dualizability and Monads}
In this section, we record how left-right duality interacts with the category of modules
over colimit preserving monads. Let $\mathscr{V}$ be a symmetric monoidal compactly generated stable category,
such that the compact objects are the same as the dualizable objects. Let $\mathscr{X}$ be a dualizable
category in $\Mod{\mathscr{V}}^L$ and
\[T : \mathscr{X} \to \mathscr{X}\]
be a colimit-preserving $\mathscr{V}$-linear monad on $\mathscr{X}$.

\begin{theorem}
  The functor which takes the pair $(T,\mathscr{X})$
  to the category
  \[\Mod{T}(\mathscr{X})\]
  is symmetric monoidal.
\end{theorem}
\begin{proof}
  As in \cite{RVCat} and known in some form since \cite{SSCat}, monads in $\Mod{\mathscr{V}}^L$ are given by
  $2$-functors
  \[\mathfrak{mnd} \to \Mod{\mathscr{V}}^L\]
  There is another $2$-category $\mathfrak{adj}$, such that $2$-functors
  \[\mathfrak{adj} \to \Mod{\mathscr{V}}^L\]
  classify adjunctions. Therefore, as $\Mod{\mathscr{V}}^L$ is a symmetric monoidal $2$-category, it
  induces a symmetric monoidal product on monads and adjunctions in $\Mod{\mathscr{V}}^L$.
  Now because of the inclusion
  \[\mathfrak{mnd} \to \mathfrak{adj}\]
  there is a natural symmetric monoidal functor which associates to an adjunction a monad
  \[\Hom(\mathfrak{adj},\Mod{\mathscr{V}}^L) \to \Hom(\mathfrak{mnd},\Mod{\mathscr{V}}^L)\]
  This functor has a lax symmetric monoidal right adjoint
  \[\Hom(\mathfrak{mnd},\Mod{\mathscr{V}}^L) \to \Hom(\mathfrak{adj},\Mod{\mathscr{V}}^L)\]
  which associates to a monad its category of modules (see also Remark 5.7 in \cite{monad}).
  This is the functor we wish to show
  is symmetric monoidal. 
  
  It is obvious the functor preserves units. As there is clearly a map
  \[\bigotimes{\Mod{T_i}(\mathscr{X}_i)} \to \Mod{\left(\bigotimes{T_i}\right)}(\bigotimes{\mathscr{X}_i})\]
  coming from the fact that the functor is lax symmetric monoidal,
  it suffices to show this map is an isomorphism. By induction we reduce to showing 
  \[\Mod{T_1}(\mathscr{X}_1) \otimes_{\mathscr{V}} \Mod{T_2}(\mathscr{X}_2) \xrightarrow{\cong} \Mod{(T_1 \otimes T_2)}(\mathscr{X}_1 \otimes_{\mathscr{V}} \mathscr{X}_2)\]
  This can be shown by Lurie-Barr-Beck (Theorem 4.7.3.5 in \cite{HA})
  if we can show that the functor
  \[G_1 \otimes G_2 : \Mod{T_1}(\mathscr{X}_1) \otimes_{\mathscr{V}} \Mod{T_2}(\mathscr{X}_2) \to \mathscr{X}_1 \otimes_{\mathscr{V}} \mathscr{X}_2\]
  (where the $G_i$'s are the forgetful functors) is conservative.
  By Theorem 4.8.4.6 in \cite{HA}, we have
  \[\Mod{T_1}(\mathscr{X}_1) \cong \Mod{T_1}(\Hom_{\mathscr{V}}(\mathscr{X}_1,\mathscr{X}_1)) \otimes_{\Hom_{\mathscr{V}}(\mathscr{X}_1,\mathscr{X}_1)} \mathscr{X}_1\]
  Hence, we have the isomorphism (using Theorem 4.8.5.16 of \cite{HA})
  \[\Mod{T_1}(\mathscr{X}_1) \otimes_{\mathscr{V}} \Mod{T_2}(\mathscr{X}_2) \cong \Mod{(T_1 \otimes \id)}(\mathscr{X}_1 \otimes_{\mathscr{V}} \Mod{T_2}(\mathscr{X}_2))\]
  So it suffices to check that the functor
  \[\mathscr{X}_1 \otimes_{\mathscr{V}} \Mod{T_2}(\mathscr{X}_2) \to \mathscr{X}_1 \otimes_{\mathscr{V}} \mathscr{X}_2\]
  is conservative. But here we can apply the same argument again \footnote{This argument
  is adapted from the proof of Theorem 4.8.5.16 in \cite{HA}}.
\end{proof}
\begin{corollary} \label{appendA}
  If $\mathscr{X}$ is a dualizable $\mathscr{V}$-module category, then for any $T$
  a $\mathscr{V}$-linear colimit preserving monad on $\mathscr{X}$,
  \[\Mod{T}(\mathscr{X})\]
  is dualizable with dual
  \[\Mod{T^{\vee}}(\mathscr{X}^{\vee})\]
\end{corollary}
\begin{proof}
  As $T$ is a colimit preserving $\mathscr{V}$-linear monad on $\mathscr{X}$, we can write $T$ as 
  \[T \in \Hom_{\mathscr{V}}(\mathscr{X},\mathscr{X}) \cong \mathscr{X}^{\vee} \otimes_{\mathscr{V}} \mathscr{X}\]
  Clearly $T$ is a $(T,T)$-bimodule. Equivalently, $T$ is a $(T \otimes T^{\vee})$-module, and hence we can write
  \[T \in \Mod{(T \otimes T^{\vee})}(\mathscr{X}^{\vee} \otimes \mathscr{X}) \cong \Mod{T^{\vee}}(\mathscr{X}^{\vee}) \otimes_{\mathscr{V}} \Mod{T}(\mathscr{X})\]
  This defines a map
  \[T: \mathscr{V} \to \Mod{T^{\vee}}(\mathscr{X}^{\vee}) \otimes_{\mathscr{V}} \Mod{T}(\mathscr{X})\]
  Now, by Theorem 4.8.4.6 in \cite{HA}, we have the isomorphism
  \[\Mod{T^{\vee}}(\mathscr{X}^{\vee}) \cong \Mod{T^{\vee}}(\Hom_{\mathscr{V}}(\mathscr{X}^{\vee},\mathscr{X}^{\vee})) \otimes_{\Hom_{\mathscr{V}}(\mathscr{X}^{\vee},\mathscr{X}^{\vee})} \mathscr{X}^{\vee}\]
  However,
  \[\Hom_{\mathscr{V}}(\mathscr{X}^{\vee},\mathscr{X}^{\vee}) \cong \mathscr{X} \otimes \mathscr{X}^{\vee} \cong \Hom_{\mathscr{V}}(\mathscr{X},\mathscr{X})\]
  is an isomorphism of categories which reverses the monoidal structure and identifies $T^{\vee}$ with $T$. Therefore, we also have the isomorphism
  \[\Mod{T^{\vee}}(\mathscr{X}^{\vee}) \cong  \mathscr{X}^{\vee} \otimes_{\Hom_{\mathscr{V}}(\mathscr{X},\mathscr{X})} \RMod{T}(\Hom_{\mathscr{V}}(\mathscr{X},\mathscr{X})) \cong \RMod{T}(\mathscr{X}^{\vee})\]
  so it is isomorphic to the category of right modules over the monad $T$ on $\mathscr{X}^{\vee}$ ($\RMod{T}$ here means right $T$ modules).
  Hence, there is a map, coming from tensor product over the monad $T$, 
  \[\otimes_T: \Mod{T^{\vee}}(\mathscr{X}^{\vee}) \otimes_{\mathscr{V}} \Mod{T}(\mathscr{X}) \to \mathscr{V}\]
  By a standard argument these form unit and counit maps, witnessing the dualizability of $\Mod{T}(\mathscr{X})$.
\end{proof}
\begin{corollary} \label{appendB}
  Suppose
  \[F_T: \mathscr{X} \to \Mod{T}(\mathscr{X})\]
  is the free $T$-module functor and
  \[G_T : \Mod{T}(\mathscr{X}) \to \mathscr{X}\]
  is the forgetful functor.
  Then
  \[(F_T)^{\vee} \cong G_{T^{\vee}}\]
  and
  \[(G_T)^{\vee} \cong F_{T^{\vee}}\]
\end{corollary}
\begin{proof}
  Direct calculation from the unit and counit maps above.
\end{proof}

\section{Crystals on Underived Noetherian Schemes} \label{cryss}
In this section, we recall the definition of a crystal on the infinitesimal site and prove an equivalence of categories between
quasi-coherent crystals on the big and small infinitesimal site.
Let $S$ be a underived Noetherian scheme. Denote by $SCH^{ft}_{/S}$ the category
of underived finite-type schemes over $S$. Suppose $X \in SCH^{ft}_{/S}$ and the structure map
$p_X : X \to S$ is finite tor-amplitude and separated.
\begin{definition}
  The big infinitesimal site $INF(X/S)$ has as objects diagrams
  \begin{equation} \label{diainf}
    \begin{tikzcd}
      U \arrow{r}{u} \arrow{d}{b} & X \arrow{d} \\
      T \arrow{r} & S\\
    \end{tikzcd}
  \end{equation}
  
  in $SCH^{ft}_{/S}$ such that $b$ is a thickening--a closed immersion of underived schemes inducing a homeomorphism.
  Morphisms in $INF(X/S)$ are defined in the obvious way.
  A family of morphisms in $INF(X/S)$,  $\{(U_i \to T_i) \to (U \to T)\}$ is a Zariski (resp. étale) covering
  if each
  \[
    \begin{tikzcd}
        U_i \arrow{r} \arrow{d} & U \arrow{d} \\
        T_i \arrow{r} & T\\
    \end{tikzcd}
  \]
  is a pullback square and the maps $\{T_i \to T\}$ is a Zariski (resp. étale) covering.
\end{definition}

The assignment $(U \to T) \mapsto \QC{T}$ defines a (Zariski or étale) sheaf of categories on $INF(X/S)$ where
the transition maps are given by quasicoherent pullback.

\begin{definition}
  The small infinitesimal site $Inf(X/S)$ is the full subcategory of $INF(X/S)$ consisting of those objects such
  that the map $u$ (in the notation of (\ref{diainf})) is an open immersion. It is also
  endowed with either the Zariski or étale topology induced from the big site.
\end{definition}

\begin{definition}
  A quasicoherent crystal on the big infinitesimal site $Inf(X/S)$ is an object of the category \[\lim_{\opposite{INF(X/S)}}\QC{T}\]
  We will call this category $CRYS(X/S)$. Similarly we can define the category of quasicoherent
  crystals on the small infinitesimal site
  \[Crys(X/S) := \lim_{\opposite{Inf(X/S)}} \QC{T}\]
\end{definition}
\begin{remark}
  Unwinding the definitions, it is clear that
  \[CRYS(X/S) \cong QCoh((X/S)_{dR})\]
  in the notation of Definition \ref{derhamstack}
\end{remark}

Note that the definition of a quasicoherent crystal does not make use of the topology at all.

\begin{theorem}
  There is an equivalence of categories
  \[Res: CRYS(X/S) \cong Crys(X/S)\]
  induced by the natural restriction functor.
\end{theorem}
\begin{proof}
  By Zariski descent we may assume $X$ and $S$ are affine. Also by Zariski descent, we have the isomorphisms
  \[Crys(X/S) \cong \lim_{\opposite{(Inf(X/S)^{aff})}} \QC{T}\]
  \[CRYS(X/S) \cong \lim_{\opposite{(INF(X/S)^{aff})}} \QC{T}\]
  where the superscript $aff$ means we restrict to considering only pairs $(U,T)$ which are both affine.
  For the claim then, it is enough to show that the inclusion functor
  \[Inc : Inf(X/S)^{aff} \to INF(X/S)^{aff}\]
  is cofinal. By \cite{HTT} Theorem 4.1.3.1 
  it suffices to show that for any $(U,T) \in INF(X/S)^{aff}$, the comma category
  \[(U,T) \downarrow Inf(X/S)^{aff}\]
  is weakly contractible (its groupoid completion is contractible). From
  now on, we fix $(U,T) \in INF(X/S)^{aff}$.
  By Lemma 5.3.1.18 (and taking opposite categories), it suffices to show that the comma category is cofiltered.
  Because the category is discrete, it suffices to the following conditions.
  \begin{enumerate}
  \item The category is nonempty.
  \item For any two objects $c_1,c_2$, there is an object $c$ with maps to both $c_1$ and $c_2$.
  \item For any two morphisms $f_1,f_2 \in \Hom(c_1,c_2)$, there exists a map $h \in \Hom(c_2,c_3)$ (for some $c_3$) such that $hf_1=hf_2$.
  \end{enumerate}
  We may embed $X$ inside $Y' := \mathbb{A}^n_S$ as a closed subscheme. Hence, we can map $T$ to a finite thickening
  of $X$ inside $Y'$ showing the first condition.
  For the second condition, suppose there are two objects in $Inf(X/S)^{aff}$, $(V_1,D_1)$ and $(V_2,D_2)$,
  both receiving maps from $(U,T)$. Now, without loss of generality we can assume $V_1 = V_2 = V$ as open subschemes of $X$.
  Then it's clear that $T$ maps to the (underived) product $D_1 \times_S D_2$, and thus to a finite thickening of $V$ diagonally
  embedded in $D_1 \times_S D_2$, showing the second condition. For the third condition, consider two maps
  \[f_1,f_2 : (V_1,D_1) \to (V_2,D_2)\]
  which are equalized by a map
  \[g: (U,T) \to (V_1,D_1)\]
  Again we assume $V_1 = V_2 = V$.
  We can consider the equalizer $D_3$ of the two maps from $D_1$ to $D_2$, which is a closed subscheme of $D_1$.
  Then, $(V,D_3)$ with the obvious compatibilities proves the third condition. Hence, the theorem follows.
\end{proof}
\begin{remark}
  It is also possible to prove the above theorem by exhibiting a hypercovering in the small site and using the
  fact that $QCoh$ is a hypercomplete sheaf.
\end{remark}

\bibliography{ref}{}
\bibliographystyle{amsalpha}
\end{document}